\documentclass[11pt]{article}

\usepackage{amsmath}
\usepackage{mathtools}
\usepackage{amssymb}
\usepackage{here}
\usepackage{mathrsfs}
\usepackage{enumerate}
\usepackage{multirow}
\usepackage{algpseudocode}
\usepackage{algorithm}
\usepackage{subfigure}
\usepackage{fullpage}
\usepackage{color}
\usepackage{cite}
\usepackage{dsfont}

\definecolor{lblue}{RGB}{0,110,152}

\definecolor{dred}{RGB}{171,67,53}

\everymath{\displaystyle}

\newtheorem{theorem}{Theorem}%[\arabic{section}]
\newtheorem{corollary}[theorem]{Corollary}
\newtheorem{proposition}[theorem]{Proposition}
\newtheorem{lemma}[theorem]{Lemma}

\newtheorem{remark}[theorem]{Remark}

\newtheorem{property}[theorem]{Property}

\newcommand{\mendth}{\hfill \ensuremath{\vartriangle}}

\DeclareMathOperator*{\col}{col}

\DeclareMathOperator{\He}{Sym}
\DeclareMathOperator*{\diag}{diag}

\DeclareMathOperator{\eps}{\varepsilon}

\DeclareMathOperator{\E}{ \mathbb{E}}

\def\P{\mathbb{P}}

\def\E{\mathbb{E}}

\providecommand\X[1]{\boldsymbol{X_{#1}}}

\providecommand\phib{\boldsymbol{\emptyset}}

\providecommand{\blue}[1]{\color{black}{#1}\color{black}\hspace{0pt}}

\providecommand\X[1]{\boldsymbol{X_{#1}}}

\providecommand\phib{\boldsymbol{\emptyset}}

\newenvironment{proof}{{\it Proof :~}}{\hfill$\diamondsuit$\\}

\begin{document}

\title{In-Silico Proportional-Integral Moment Control of Stochastic Gene Expression}

\author{Corentin Briat and Mustafa Khammash\thanks{email: corentin@briat,info; mustafa.khammash@bsse.ethz.ch; url: www.briat.info; www.bsse.ethz.ch/ctsb}\\D-BSSE, ETH-Z\"urich}
\date{}

%\markboth{IEEE TRANSACTIONS ON AUTOMATIC CONTROL, VOL. XX, NO. X, XX XXXX}%
%{In-Silico Proportional-Integral Moment Control of Stochastic Reaction Networks}

\maketitle

\begin{abstract}
The problem of controlling the mean and the variance of a species of interest in a simple gene expression is addressed. It is shown that the protein mean level can be globally and robustly tracked to any desired value using a simple PI controller that satisfies certain sufficient conditions. Controlling both the mean and variance however requires an additional control input, e.g. the mRNA degradation rate, and local robust tracking of mean and variance is proved to be achievable using multivariable PI control, provided that the reference point satisfies necessary conditions imposed by the system. Even more importantly, it is shown that there exist PI controllers that locally, robustly and simultaneously stabilize all the equilibrium points inside the admissible region. The results are then extended to the mean control of a gene expression with protein dimerization. It is shown that the moment closure problem can be circumvented without invoking any moment closure technique. Local stabilization and convergence of the average dimer population to any desired reference value is ensured using a pure integral control law. Explicit bounds on the controller gain are provided and shown to be valid for any reference value. As a byproduct, an explicit upper-bound of the variance of the monomer species, acting on the system as unknown input due to the moment openness, is obtained. The results are illustrated by simulation.

\end{abstract}

%\IEEEpeerreviewmaketitle

\section{Introduction}

The regulation problem of single-cells and \blue{the regulation problem of} cell populations have attracted a lot of attention \blue{in} recent years because of their potential in applications such as in bioreactors \cite{Chubukov:16,Briat:17ACS}, disease treatment (homeostasis restoring \cite{Schukur:16}), etc. Various ways for controlling single cells or cell populations exist. \blue{Among these} is the so-called in-vivo control where controllers are implemented inside cells using biological elements (genes, RNA, proteins, etc.) using synthetic biology and bioengineering techniques. Theoretical works have suggested that certain motifs could achieve perfect adaptation for the controlled network. Notably, circuits relying on an annihilation reaction of the controller species have proven to be a credible solution to the regulation problem; see e.g. \cite{Klavins:10,Briat:15e,Qian:17,Olsman:17,Annunziata:17}. Other designs involve, for instance, phosphorylation cycles exhibiting an ultrasensitive behavior \cite{Cuba:17} or autocatalytic reactions \cite{Briat:16a,Briat:19:Logistic}. The alternative to in-vivo control is the so-called in-silico control in which control functions are moved outside the cells and are relegated to a digital computer \cite{Khammash:11,Uhlendorf:12,Menolascina:14,Fiore:16,Milias:16,Lugagne:17}. Actuation can be performed using light (optogenetics) or using chemical inducers (microfluidics) whereas measurements are mostly performed via the use of fluorescent molecules (such as the Green Fluorescent Protein -- or GFP) whose copy number can be estimated using time-lapse microscopy or flow cytometry.

In the current paper, we focus on the in-silico control of cell populations in the moments equation framework, for which we summarize and extend some of our previously obtained theoretical results. The paper is divided in three parts. The first part is devoted to the control of the mean protein copy number in a simple gene expression network across a cell population using a Proportional-Integral (PI) control strategy. The considered control input is the transcription rate of the mRNA and the measured output is the mean of the protein copy number across the cell population. Note that this measure is perfectly plausible as it can be computed from the fluorescence (or the protein copy number) distribution across the cell population. As the output of a PI controller can be negative, which would not correspond to any meaningful control input here, we resolve this problem by placing an ON/OFF nonlinearity \cite{Goncalves:00,Goncalves:07} between the controller and the system. Another solution to this problem is the use of an antithetic integral controller or positively regularized integral controllers; see e.g. \cite{Briat:15e,Briat:19:Positive}. The local exponential stability can be easily proven using a standard eigenvalue analysis, whereas the global  asymptotic stability is proved in the context of absolute stability using Popov' criterion; see e.g. \cite{Popov:61,Briat:12c,Guiver:15}. Additional results on the disturbance rejection, the robustness with delays and the generalization to arbitrary moment equations are also given.

%some of which were absent from the conference articles underlying this regular paper.

The second part is devoted to the extension of the above problem where, \blue{in addition to controlling the protein mean}, we also aim at controlling the variance in the protein copy number across the cell population. We first show that a second control input, the degradation rate of the mRNA, needs to be considered in order to able to independently control the mean and the variance of the protein copy number across the cell population. We then identify the existence of a lower bound for the stationary variance that cannot be overcome by any in-silico control strategy using the moments as measured variables. This has to be contrasted with what is achievable using in-vivo control where the variance can be theoretically reduced below this value \cite{Briat:15e,Briat:18:Interface} by suitably adjusting a negative feedback gain. Using the transcription rate and the degradation rate of the mRNA as control inputs, we demonstrate that a multivariable PI feedback can be used to locally and robustly track any desired protein mean and variance set-point, provided this set-points satisfies a certain necessary condition imposed by the structure of the system. %\red{It is also shown that there exists a  multivariable PI controller that locally and robustly stabilizes any desired admissible mean and variance values. Finally, numerical simulations illustrate the effectiveness of the designed genetic control systems.}

The third and last part of the paper is devoted to the control of the mean copy number of a protein dimer in a gene expression network with protein dimerization. Although \blue{this may seem} incremental, this network exhibits a considerable increase in its complexity \blue{due to the dimerization}. The main difficulty lies in the fact that the moments equations are \blue{no longer closed}, which forces us to work with a moment equation having some unknown, yet potentially measurable, inputs. A second difficulty is that the moment equation becomes nonlinear, which increases the complexity of the problem. A natural approach would be to close the moments using some moment closure method (see e.g. \cite{Hespanha:08b}) and to control the closed moment equations. However, because of the inaccuracy of moment closure schemes it is not guaranteed that the original moment equation will remain stable under the same conditions as the closed moments equation. We propose here an alternative approach whereby we exploit the ergodicity of the process and solve the control problem directly \blue{for} the open moment equations where the variance of the protein copy number acts as an input. We show that this input is necessarily bounded by some value that we explicitly characterize. Some local asymptotic stability conditions are also obtained. This approach suggests that the moment equation openness is not as problematic as in the simulation/analysis problem. \blue{Finally, simulated examples} are given for illustration. \\

\noindent\textbf{Outline.} Preliminary definitions and results are given in Section \ref{sec:prel}. The problem of controlling the mean protein copy number in a gene expression network is addressed in Section \ref{sec:mean} whereas the problem of controlling both the mean and the variance in the protein copy number is considered in Section \ref{sec:var}. Finally, the mean control of the dimer copy number in a gene expression network with dimerization is treated in Section \ref{sec:dimer}. Examples are given in the related sections.\\

\noindent\textbf{Notations.} The notation is pretty much standard. Given a random variable $X$, its expectation is denoted by $\E[X]$. For a square matrix $M$, $\He[M]$ stands for the sum $M+M^*$. Given a vector $v\in\mathbb{R}^n$, the notation $\diag(v)$ stands for a diagonal matrix having the elements of $v$ as diagonal entries.

\section{Preliminaries on stochastic reaction networks}\label{sec:prel}

We briefly introduce in this section stochastic reaction networks as well as essential models. For more details, the readers are referred to \cite{Anderson:15}. A stochastic reaction network is a system where $N$ molecular species $\X{1},\ldots,\X{N}$ interact with each others through $M$ reaction channels $R_1,\ldots,R_M$. Each reaction $R_k$ is described by a propensity function $w_k:\mathbb{Z}_{\ge0}\mapsto\mathbb{R}_{\ge0}$ and a stoichiometric vector $s_k\in\mathbb{Z}^N$. The propensity function $w_k$ is such that for all $\varkappa\in\mathbb{Z}_{\ge0}$ such that $\varkappa+s_k\notin\mathbb{Z}_{\ge0}$, we have that $w_k(\varkappa)=0$. Assuming homogeneous mixing, it can be shown that the process $(X_1(t),\ldots,X_N(t))_{t\ge0}$ is a Markov process that can be described by the so-called Chemical Master Equation (CME), or Forward Kolmogorov equation, given by
\begin{equation*}
  \dot{p}_{\varkappa_0}(\varkappa,t)=\sum_{k=1}^M\left[w_k(\varkappa-s_k)p_{\varkappa_0}(\varkappa-s_k,t)-w_k(\varkappa)p_{\varkappa_0}(\varkappa,t)\right]
\end{equation*}
where $p_{\varkappa_0}(\varkappa,t)=\P[X(t)=\varkappa|X(0)=\varkappa_0]$ and $\varkappa_0$ is the initial condition. Based on the CME, dynamical expressions for the first- and second-order moments may be easily derived and are given by
\begin{equation}\label{eq:moments_g}
\hspace{-3mm}\begin{array}{lcl}
    \dfrac{d \E[X(t)]}{dt}&=&S\E[w(X(t))],\\
    \dfrac{d \E[X(t)X(t)^T]}{dt}&=&\He[S\E[w(X(t))X(t)^T]]\\
    &&+S\diag\{\E[w(X(t))]\}S^T
\end{array}
\end{equation}
where $S:=\begin{bmatrix}
  s_1 & \ldots & s_M
\end{bmatrix}\in\mathbb{R}^{N\times M}$ is the stoichiometry matrix and $w(\varkappa):=\begin{bmatrix}
  w_1(\varkappa) & \ldots & w_M(\varkappa)
\end{bmatrix}^T\in\mathbb{R}^{M}$ the vector of propensity functions. When the propensity functions are affine (i.e. $w(\varkappa)=W\varkappa+w_0$, $W\in\mathbb{R}^{M\times N}$, $w_0\in\mathbb{R}^M$), then the above system can be rewritten in the form
\begin{equation}\label{eq:moments}
\begin{array}{rcl}
    \dfrac{d \E[X(t)]}{dt}&=&SW\E[X(t)]+Sw_0,\\
    \dfrac{d\Sigma(t)}{dt}&=&\He[SW\Sigma(t)]+S\diag(W\E[X(t)]+w_0)S^T
\end{array}
\end{equation}
where $\Sigma(t):=\E[(X(t)-\E[X(t)])(X(t)-\E[X(t)])^T]$ is the covariance matrix.  An immediate property of \eqref{eq:moments} is that it forms a closed system, unlike in the more general case \eqref{eq:moments_g} where the first and second order moments depend on higher-order ones whenever the propensity functions are of mass-action.

\section{Mean control of protein levels in a gene expression network}\label{sec:mean}

The objective of the current section is to give a clear picture of the mean control of the number of proteins using a simple \emph{positive PI controller}, i.e. a PI controller generating nonnegative control inputs. The considered control input is the transcription rate $k_r$ which can be externally actuated using, for instance, light-induced transcription \cite{Khammash:11}. It is shown in this section that a positive PI control law allows to achieve global and robust output tracking of the mean number of proteins.

\subsection{Preliminaries}

We consider in this section the following model for gene expression
\begin{equation}\label{eq:reacnet_gene}
\begin{array}{lccclclcccl}
   R_1&:&\phib&\stackrel{k_r}{\longrightarrow}&\X{1},&&  R_2&:&\X{1}&\stackrel{\gamma_r}{\longrightarrow}&\phib,\\
   R_3&:&\X{1}&\stackrel{k_p}{\longrightarrow}&\X{1}+\X{2},&&  R_4&:&\X{2}&\stackrel{\gamma_p}{\longrightarrow}&\phib
 \end{array}
\end{equation}
where $\X{1}$ denotes the mRNA and $\X{2}$ the associated protein. %The stoichiometry matrix associated with the gene expression network is given by
%\begin{equation}
%    S=\begin{bmatrix}
%      1 & -1 & 0 & 0\\
%      0 & 0 & 1 & -1
%    \end{bmatrix}
%\end{equation}
%and, assuming mass-action kinetics, the vector of propensity functions is given by
%\begin{equation}
%    w(X)=\begin{bmatrix}
%      k_r & \gamma_rX_1 &  k_pX_1 & \gamma_pX_2
%    \end{bmatrix}^T.
%\end{equation}
%%
In vector form, the equations \eqref{eq:moments} rewrite
\begin{equation}\label{eq:mainsyst}
\left[\begin{array}{c}
   %\begin{bmatrix}
    \dot{x}_1(t)\\
    \dot{x}_2(t)\\
    \hline
    \dot{x}_{3}(t)\\
    \dot{x}_{4}(t)\\
    \dot{x}_{5}(t)
  %\end{bmatrix}
  \end{array}\right]=\left[\begin{array}{c|c}
    A_{ee} & 0\\
    \hline
    A_{\sigma e} & A_{\sigma\sigma}
  \end{array}\right]\left[\begin{array}{c}
    x_1(t)\\
    x_2(t)\\
    \hline
     x_{3}(t)\\
    x_{4}(t)\\
    x_{5}(t)
  \end{array}\right]+\left[\begin{array}{c}
    B_e\\
    \hline
    B_\sigma
  \end{array}\right]k_r
\end{equation}
where the state variables are defined by $x_i:=\E[X_i]$, $i=1,2$, $x_3:=V(X_1)$, $x_4:=\textnormal{Cov}(X_1,X_2)$, $x_5:=V(X_2)$,
%\begin{equation*}
%  \begin{bmatrix}
%    x_1\\
%    x_2
%  \end{bmatrix}:=\E[X]\ \text{and}\ \begin{bmatrix}
%    x_3 & x_4\\
%    x_4 & x_5
%  \end{bmatrix}:=\Sigma
%\end{equation*}
%
and the system matrices by
\begin{equation}\label{eq:mainsystm}
\begin{array}{lcl}
   A_{ee}&=&\begin{bmatrix}
    -\gamma_r & 0\\
    k_p & -\gamma_p
  \end{bmatrix},\ A_{\sigma e}=\begin{bmatrix}
    \gamma_r & 0\\
    0 & 0\\
    k_p & \gamma_p
  \end{bmatrix},\ B_e=\begin{bmatrix}
    1\\0
  \end{bmatrix},\\
  A_{\sigma\sigma}&=&\begin{bmatrix}
    -2\gamma_r & 0 & 0\\
    k_p & -(\gamma_r+\gamma_p) & 0\\
    0 & 2k_p & -2\gamma_p
  \end{bmatrix},\ B_\sigma=\begin{bmatrix}
    1\\
    0\\
    0
  \end{bmatrix},
\end{array}
\end{equation}
where $k_r>0$ is the transcription rate of DNA into mRNA, $\gamma_r>0$ is the degradation rate of mRNA, $k_p>0$ is the translation rate of mRNA into protein and $\gamma_p>0$ is the degradation rate of the protein.

We consider here the following restriction of system (\ref{eq:mainsyst}):
\begin{equation}\label{eq:musyst}
  \begin{bmatrix}
    \dot{x}_1(t)\\
    \dot{x}_2(t)
  \end{bmatrix}=A_{ee}\begin{bmatrix}
    x_1(t)\\
    x_2(t)
  \end{bmatrix}+B_eu(t)
\end{equation}
and the positive PI control law
\begin{equation}\label{eq:pi1}
  u(t)=\varphi\left(k_1(\mu_*-x_2(t))+k_2\int_0^t[\mu_*-x_2(s)]ds\right)
\end{equation}
where $\mu_*$ is the mean number of protein to track and the scalars $k_1,k_2$ the gains of the controller. The nonlinearity $\varphi(\cdot)$ can either be an ON/OFF nonlinearity $\varphi(u):=\max\{0,u\}$ \cite{Goncalves:07} or a saturated ON/OFF one $\varphi(u):=\min\{\max\{0,u\},\bar u\}$ where $\bar u>0$ is the maximum admissible value for the control input.

\begin{property}
  Given a constant reference $\mu_*\ge0$, the equilibrium point of the system (\ref{eq:musyst})-(\ref{eq:pi1}) is given by $x_2^*=\mu_*$ and
  \begin{equation}\label{eq:mupt}
     x_1^*=\dfrac{\mu_*\gamma_p}{k_p},\ u^*=\dfrac{\mu_*\gamma_p\gamma_r}{k_p}\text{\ and\ } I^*=\frac{u^*}{k_2}
  \end{equation}
  where $I^*$ is the equilibrium value of the integral term. When a saturated ON/OFF nonlinearity is used, for this equilibrium point to be meaningful, we need that $u^*\le\bar u$.
\end{property}
As expected, the integrator allows to achieve constant set-point tracking for the mean protein count regardless the values of the parameters of the system.

\subsection{Local stabilizability, stabilization and output tracking}

Since the equilibrium control input $u^*$ and the set-point $\mu_*$ are both positive, the nonlinearity $\varphi(\cdot)$ is not active in a sufficiently small neighborhood of the equilibrium point (\ref{eq:mupt}). The ON/OFF nonlinearity can hence be locally ignored and the local analysis performed on the corresponding linear system. Assuming first that the system parameters $k_p,\gamma_p,\gamma_r$ are exactly known, the following result on local nominal stabilizability and stabilization is obtained:
\begin{lemma}\label{lem:jdksjdlks}
  Given system parameters $k_p,\gamma_p,\gamma_r>0$, the system (\ref{eq:musyst}) is locally stabilizable using the control law (\ref{eq:pi1}).  Moreover, the equilibrium point (\ref{eq:mupt}) of the closed-loop system (\ref{eq:musyst})-(\ref{eq:pi1}) is locally asymptotically (exponentially) stable if and only if the conditions
  \begin{equation}\label{eq:nomcond}
     \blue{\dfrac{\mu_*\gamma_p\gamma_r}{k_p}\le\bar{u},}\ k_1>\dfrac{k_2}{\gamma_p+\gamma_r}-\dfrac{\gamma_p\gamma_r}{k_p},\text{ and }k_2>0
  \end{equation}
  hold.\mendth
\end{lemma}
\begin{proof}
  The closed-loop system (\ref{eq:musyst})-(\ref{eq:pi1}) is given by
  \begin{equation}\label{eq:aug}
\begin{bmatrix}
  \dot{x}_1(t)\\
  \dot{x}_2(t)\\
  \dot{I}(t)
\end{bmatrix}=\begin{bmatrix}
    -\gamma_r & -k_1 & k_2\\
    k_p & -\gamma_p & 0\\
    0 & -1 & 0
  \end{bmatrix}\begin{bmatrix}
  x_1(t)\\
  x_2(t)\\
  I(t)
\end{bmatrix}+\begin{bmatrix}
  k_1\\0\\1
\end{bmatrix}\mu_*
  \end{equation}
where $I$ is the integrator state of the controller. Local stabilizability is then equivalent to the existence of a pair $(k_1,k_2)\in\mathbb{R}^2$ such that the state matrix of the augmented system (\ref{eq:aug}) is Hurwitz stable, i.e. has poles in the open left-half plane. The Routh-Hurwitz criterion yields the conditions (\ref{eq:nomcond}) that define a nonempty subset of the plane $(k_1,k_2)$. System (\ref{eq:musyst}) is hence locally stabilizable using the PI control law (\ref{eq:pi1}) for any triplet of parameter values ${(k_p,\gamma_r,\gamma_p)\in\mathbb{R}_{>0}^3}$. As a consequence, the closed-loop system is locally asymptotically stable when the control parameters are located inside the stability region defined by the conditions (\ref{eq:nomcond}).
\end{proof}

In order to extend the above result to the uncertain case, we assume here that the system parameters $(k_p,\gamma_p,\gamma_r)$ belong to the set
\begin{equation}\label{eq:paramset}
  \mathcal{P}_\mu:=[k_p^-,k_p^+]\times[\gamma_p^-,\gamma_p^+]\times [\gamma_r^-,\gamma_r^+]
\end{equation}
where the parameter bounds $k_p^-\le k_p^+,\gamma_p^-\le \gamma_p^+$ and $\gamma_r^-\le \gamma_r^+$ are real positive numbers. We then obtain the following result:
\begin{lemma}\label{lem:robloc}
  The system (\ref{eq:musyst}) with uncertain constant parameters $(k_p,\gamma_p,\gamma_r)\in\mathcal{P}_\mu$ is robustly locally asymptotically (exponentially) stabilizable using the control law (\ref{eq:pi1}). Moreover, the equilibrium point (\ref{eq:mupt}) of the closed-loop system (\ref{eq:musyst})-(\ref{eq:pi1}) is locally robustly asymptotically (exponentially) stable if and only if the conditions
  \begin{equation}\label{eq:robcond}
    \blue{\dfrac{\mu_*\gamma_p^+\gamma_r^+}{k_p^-}\le\bar{u}},\ k_1>\dfrac{k_2}{\gamma_p^-+\gamma_r^-}-\dfrac{\gamma_r^-\gamma_p^-}{k_p^+},\text{ and }k_2>0
  \end{equation}
  hold. \blue{When $\bar{u}=\infty$, then $\gamma_p^+$ and $\gamma_r^+$ can be arbitrarily large and $k_p^-$ can be arbitrarily small. }  \mendth
\end{lemma}
\begin{proof}
\blue{The starting point is the lower bound for $k_1$ obtained in Lemma \ref{lem:jdksjdlks} given by $f(x,y,z):=\frac{k_2}{y+z}-\frac{yz}{x}$, $x,y,z,k_2>0$ and let ${\bar{f}:=\sup_{(x,y,z)\in\mathcal{P}_\mu}f(x,y,z)}$ such that $\mu_*yz/x\le\bar{u}$ for all $(x,y,z)\in\mathcal{P}_\mu$. Simple calculations show that $f(x,y,z)$ is increasing in $x$ and decreasing in $y,z$ over $(x,y,z)\in\mathcal{P}_\mu$. Hence, we have $\bar{f}=f(k_p^+,\gamma_p^-,\gamma_r^-)$ and $k_1>\bar{f}$ implies that $k_1>f(x,y,z)$ for all $(x,y,z)\in\mathcal{P}_\mu$. This concludes the proof.}
   \end{proof}

\subsection{Global stabilizability, stabilization and output tracking}

Local properties obtained in the previous section are generalized here to global ones.  Noting first that the nonlinear function $\varphi(\cdot)$ is time-invariant and belongs to the sector $[0,1]$, i.e. $0\le\varphi(x)/x\le 1$, $x\in\mathbb{R}$, stability can then be analyzed using absolute stability theory \cite{Khalil:02} and an extension of the Popov criterion \cite{Popov:61,Khalil:02} for marginally stable systems \cite{Jonsson:97,Fliegner:06}. We have the following result:
%
%\begin{theorem}
%   Given system parameters $k_p,\gamma_p,\gamma_r>0$ and assume $k_1\ge0$, the following statements are equivalent:
%       \begin{enumerate}
%      \item The equilibrium point  (\ref{eq:mupt}) of the closed-loop system (\ref{eq:musyst})-(\ref{eq:pi1}) is locally exponentially stable.
%      \item The equilibrium point  (\ref{eq:mupt}) of the closed-loop system (\ref{eq:musyst})-(\ref{eq:pi1}) is  globally asymptotically stable.\mendth
%    \end{enumerate}
%\end{theorem}
\begin{theorem}\label{th:global}
   Given system parameters $k_p,\gamma_p,\gamma_r>0$, then the equilibrium point  (\ref{eq:mupt}) of the closed-loop system (\ref{eq:musyst})-(\ref{eq:pi1}) is globally asymptotically stable if $u^*\le\bar u$, the conditions \eqref{eq:nomcond} hold and one of the following statements hold for some $q>0$
   \begin{enumerate}[(a)]
     \item $z_0(q)>0$ and $z_1(q)>0$, or
     \item $z_0(q)>0$, $z_1(q)<0$ and $z_1(q)^2-4z_0(q)<0$
   \end{enumerate}
   where
   \begin{equation}\label{eq:djsqdj0}
\begin{array}{lcl}
    z_0(q)&=&\gamma_r^2\gamma_p^2+k_p\left[\gamma_r(\gamma_pk_1-k_2+\gamma_pk_2q)-\gamma_pk_2\right]\\
    z_1(q)&=&\gamma_p^2+\gamma_r^2+k_p\left[(k_1(\gamma_p+\gamma_r)-k_2)q-k_1\right].\hfill\mendth
  \end{array}
   \end{equation}
\end{theorem}
\begin{proof}
Accordingly to the absolute stability paradigm, the closed-loop system (\ref{eq:musyst})-(\ref{eq:pi1}) is rewritten as the \emph{negative} interconnection of the marginally stable LTI system
\begin{equation}
  H(s)=\dfrac{k_p(k_1s+k_2)}{s(s+\gamma_r)(s+\gamma_p)}
\end{equation}
and the static nonlinearity $\varphi(\cdot)$; \blue{i.e. $u=-\varphi\circ Hu$}. We assume in the following that $k_2>0$, which is a necessary condition for local asymptotic stability of the equilibrium point (\ref{eq:mupt}). Since $\lim_{s\to0}sH(s)=\frac{k_pk_2}{\gamma_r\gamma_p}>0$, then the Popov criterion \cite{Popov:61,Fliegner:06} can be applied. It states that system (\ref{eq:musyst}) is absolutely stabilizable using controller (\ref{eq:pi1}) if there exist $(k_1,k_2)\in\mathbb{R}^2$ and $q\ge0$ such that the condition
\begin{equation}\label{eq:cost}
F(j\omega,q):=\Re\left[(1+qj\omega)H(j\omega)\right]>-1
\end{equation}
holds for all $\omega\in\mathbb{R}$. In order to check this condition, first rewrite $F(j\omega,q)$ as
\begin{equation}
  F(j\omega,q)=\dfrac{N_0(\omega)}{D(\omega)}+q\dfrac{N_1(\omega)}{D(\omega)}
\end{equation}
where
\begin{equation}
  \begin{array}{lcl}
        N_0(\omega)&=&k_p\left[k_1(\gamma_r\gamma_p-\omega^2)-k_2(\gamma_r+\gamma_p)\right]\\
        N_1(\omega)&=&k_p\left[k_1\omega^2(\gamma_r+\gamma_p)+k_2(\gamma_p\gamma_r-\omega^2)\right]\\
        D(\omega)&=&(\omega^2+\gamma_r^2)(\omega^2+\gamma_p^2).
  \end{array}
\end{equation}
Since $D(\omega)>0$ for all $\omega\in\mathbb{R}$, then the condition (\ref{eq:cost}) is equivalent to
%\begin{equation}
  $N_0(\omega)+q N_1(\omega)+D(\omega)>0$
%\end{equation}
for all $\omega\in\mathbb{R}$. Letting $\bar{\omega}:=\omega^2$, we get
%\begin{equation}\label{eq:P}
  $Z(\bar{\omega}):=\bar{\omega}^2+z_1(q)\bar{\omega}+z_0(q)>0$
%\end{equation}
for all $\bar{\omega}\in[0,\infty)$ and where $z_0(q)$ and $z_1(q)$ are as in \eqref{eq:djsqdj0}. The problem therefore essentially becomes a positivity analysis of the polynomial $Z(\bar{\omega})$ over $[0,\infty)$. This polynomial is positive if and only if either $z_0(q)>0$ and $z_1(q)>0$ or $z_0(q)>0$, $z_1(q)<0$ and $z_1(q)^2-4z_0(q)<0$. To prove global asymptotic stability, it is enough to note that since $u^*>0$, we have $\varphi(u^*)=u^*$ and the control input equilibrium value does not lie in the kernel of $\varphi$. According to \cite{Fliegner:06}, this allows to conclude on the global asymptotic stability of the equilibrium point (\ref{eq:mupt}). The proof is complete.
\end{proof}
As the conditions of Theorem \ref{th:global} are implicit in nature, we provide here some more useful conditions
\begin{corollary}
   Given system parameters $k_p,\gamma_p,\gamma_r>0$, then the equilibrium point  (\ref{eq:mupt}) of the closed-loop system (\ref{eq:musyst})-(\ref{eq:pi1}) is globally asymptotically stable if the following conditions
   \begin{equation}\label{eq:djsqdj}
      \blue{\dfrac{\mu_*\gamma_p\gamma_r}{k_p}\le\bar{u},}\ k_1>\dfrac{k_2}{\gamma_p+\gamma_r}\text{ and }k_2>0
   \end{equation}
   hold.\mendth
\end{corollary}

\begin{proof}
This result can be proven by noticing that when both the conditions $k_2>0$ and $k_1(\gamma_p+\gamma_r)-k_2>0$ hold, then $z_0(q)$ and $z_1(q)$ can both be made positive provided that  $q\ge0$ is chosen sufficiently large, proving then that the equilibrium point (\ref{eq:mupt}) is globally stable when these conditions are met.
\end{proof}
\blue{\begin{remark}
  It is interesting to note that the local and global stability conditions coincides in the limit $k_p\to\infty$. This suggests that when the DC-gain of the system increases, the global conditions become less and less conservative. Note, however, that the conditions are likely to be conservative due to the fact that the sector nonlinearity includes negative values as well as any other type of nonlinearity within that sector.
\end{remark}}

It is immediate to obtain the following extension to the uncertain case:
\begin{lemma}
   Given system parameters $(k_p,\gamma_p,\gamma_r)\in\mathcal{P}_\mu$, then the equilibrium point  (\ref{eq:mupt}) of the closed-loop system (\ref{eq:musyst})-(\ref{eq:pi1}) is globally robustly asymptotically stable if the following conditions
   \begin{equation}\label{eq:djsqdj2}
     \blue{\dfrac{\mu_*\gamma_p^+\gamma_r^+}{k_p^-}\le\bar{u},}\  k_1>\dfrac{k_2}{\gamma_p^-+\gamma_r^-},\text{ and }k_2>0
   \end{equation}
   hold. \blue{When $\bar{u}=\infty$, then $\gamma_p^+$ and $\gamma_r^+$ can be arbitrarily large and $k_p^-$ can be arbitrarily small. }  \mendth
\end{lemma}
%
%Note that complementary sufficient conditions for nominal and robust stability could have been extracted from the polynomial $Z$ in (\ref{eq:P}), either by considering the cases  $k_1\gamma_p-k_2=0$ and  $k_1\gamma_p-k_2<0$; or by using Sturm series \cite{Sturm:29} that could provide necessary and sufficient conditions for the polynomial $Z$ to be positive over $\bar{\omega}\in[0,\infty)$. These conditions are however quite intricate and, for the sake of simplicity, have not been retained in the current work. It is indeed quite to difficult to draw interesting conclusions from them, unlike conditions (\ref{eq:djsqdj}) which can interpreted as a restriction of the local stability conditions (\ref{eq:nomcond}).

\subsection{Generalization of the global result to any moment equation}

We consider here an arbitrary moment equation
\begin{equation}\label{eq:momentA}
\begin{array}{rcl}
  \dot{x}(t)&=&Ax(t)+Bu(t)\\
  y(t)&=&Cx(t)
\end{array}
\end{equation}
where $A\in\mathbb{R}^n$ is Metzler, $B\in\mathbb{R}^{n\times m}_{\ge0}$ and $C\in\mathbb{R}^{m\times n}_{\ge0}$. We also assume that we have as many PI controllers than input/output pairs and that
\begin{equation}\label{eq:momentPI}
  u(t)=\varphi\left(K_1(\mu_*-y(t))+K_2\int_0^t(\mu_*-y(s))ds\right)
\end{equation}
where $K_1$ and $K_2$ are the (matrix) gains of the controller. We have the following result:
\begin{theorem}
  Let us consider the moment equation \eqref{eq:momentA} with the controllers \eqref{eq:momentPI} and assume that the set-point $(\mu_*^1,\ldots,\mu_*^m)$ is achievable (i.e. there is a positive $u^*\le\bar u$ such that $y^*=\mu_*$). Then, the unique equilibrium point of the closed-loop system \eqref{eq:momentA}-\eqref{eq:momentPI} is globally asymptotically stable if there exist a positive semidefinite matrix $N$ and an invertible matrix $Z$ such that
  \begin{equation}
    \He[I_m+(I_m+j\omega N)Z G(j\omega)Z^{-1}]>0 \textnormal{ for all }\omega\in\mathbb{R}
  \end{equation}
  where $G(s)=(K_1+K_2/s)C(sI-A)^{-1}B$. Alternatively, this is equivalent to the existence of symmetric positive definite real matrices $P$ and $\Gamma$ and a symmetric positive semidefinite matrix $\bar N$ such that the matrix
  \begin{equation}\label{eq:LMI}
    \begin{bmatrix}
      A_a^TP+PA_a & PB_a-C_a^T\Gamma-(\bar NC_aA_a)^T\\
      \star & -2\Gamma-\He[\bar{N}C_aB_a]
    \end{bmatrix}
  \end{equation}
  is negative definite where
  \begin{equation}
    \begin{bmatrix}
      A_a & \vline & B_a\\
      \hline
      C_a & \vline & 0
    \end{bmatrix}=\begin{bmatrix}
      A & 0 & \vline & B\\
      C & 0 & \vline & 0\\
      \hline
      K_1C & K_2 & \vline & 0
    \end{bmatrix}.
  \end{equation}
\end{theorem}
\begin{proof}
  This follows from the multivariable Popov criterion with the use of scalings (see e.g. \cite{Heath:09}). The LMI condition can be obtained by noticing that the frequency domain condition is equivalent to saying that the system $I_m+(I_m+sN)Z G(s)Z^{-1}$ is strictly positive real. From the Kalman-Yakubovich-Popov Lemma, this is equivalent to saying that the matrix
  \begin{equation}
    \begin{bmatrix}
      A_a^TP+PA_a & PBZ^{-1}-(ZC_a+\bar NZC_aA_a)^T\\
      \star & -2\Gamma-\He[\bar{N}ZC_aB_aZ^{-1}]
    \end{bmatrix}
  \end{equation}
  is negative definite. Performing a congruence transformation with respect to the matrix $\diag(I_n,Z)$ yields the condition \eqref{eq:LMI} where we have used the changes of variables $\bar N=Z^TNZ$ and $\Gamma=Z^TZ$. The proof is completed.
\end{proof}
While it may be difficult to analytically check these conditions in the general case, they can be numerically checked using semidefinite programming \cite{Sturm:01a}. Note, however, that when the gains $K_1$ and $K_2$ of the controller are not fixed a priori, then the problem becomes nonlinear and ad-hoc iterative methods may need to be considered to find suitable gains. Note that this problem relates to the design of a static output feedback controller, some instances being known to be NP-hard \cite{Blondel:97}.

\subsection{Input disturbance rejection}

It seems important to discuss disturbance rejection properties of the closed-loop system. For instance, basal transcription rates can be seen as constant input disturbances that need to be rejected. However, because of the positivity requirement for the control input, the rejection of constant input disturbances is only possible when they remain within certain bounds.
\begin{lemma}
  Given system parameters $k_p,\gamma_p,\gamma_r>0$, the control law (\ref{eq:pi1}) globally rejects constant input disturbances $\delta_u$ that satisfy
  \begin{equation}\label{eq:conddist1}
    \dfrac{\gamma_p\gamma_r}{k_p}\mu_*-\bar{u}\le\delta_u\le\dfrac{\gamma_p\gamma_r}{k_p}\mu_*%\left(\mu_*-\delta_y\right)
  \end{equation}
  provided that the controller gains satisfy conditions (\ref{eq:djsqdj}).\mendth
\end{lemma}
\begin{proof}
  In presence of constant input disturbances, the equilibrium value of the control input is given by
  \begin{equation}
    u_\delta^*:=\dfrac{\gamma_p\gamma_r}{k_p}\mu_*-\delta_u.%\left(\mu_*-\delta_y\right)-\delta_u.
  \end{equation}
  This value needs to be nonnegative in order to be driven by the on-off nonlinearity $\varphi$, which is the case if and only if condition (\ref{eq:conddist1}) holds.
  \end{proof}

The above result readily extends to the uncertain case:
\begin{lemma}
  Assume $(k_p,\gamma_p,\gamma_r)\in\mathcal{P}_\mu$, the control law (\ref{eq:pi1}) satisfying conditions (\ref{eq:djsqdj2}) globally and robustly rejects constant input disturbances if and only if the condition% disturbances $\delta_u$ and $\delta_y$ if and only if the condition
  \begin{equation}\label{eq:conddist2}
    \dfrac{\gamma_p^+\gamma_r^+}{k_p^-}\mu_*-\bar{u}\le\delta_u\le\dfrac{\gamma_p^-\gamma_r^-}{k_p^+}\mu_*%\left(\mu_*-\delta_y\right)
  \end{equation}
  is fulfilled.\mendth
\end{lemma}
\begin{proof}
%The proof follows from standard analysis.
  The proof follows from a simple extremum argument similar to the one used in the proof of Lemma \ref{lem:robloc}.
\end{proof}
It seems important to point out that the sets of admissible perturbations defined by (\ref{eq:conddist1}) or (\ref{eq:conddist2}) do not depend on the choice for the controller gains. They do, however, depend on the mean reference value $\mu_*$, which is expected since small $\mu_*$'s yield small control inputs which are more likely to be overwhelmed by disturbances.

\subsection{Local robustness with respect to constant input delay}

Let us consider now that the control input is delayed by some constant delay $h>0$. In this case, it is convenient to rewrite the closed-loop system according to the state variable $(x_1,x_2,u)$:
  %\begin{equation}
%    \begin{bmatrix}
%      \dot{x}_1(t)\\
%      \dot{x}_2(t)\\
%      \dot{u}(t)
%    \end{bmatrix}=\begin{bmatrix}
%      -\gamma_r & 0 & 1\\
%      k_p & -\gamma_p & 0\\
%       -k_1k_p & k_1\gamma_p-k_2 & 0
%    \end{bmatrix}\begin{bmatrix}
%     x_1(t)\\
%     x_2(t)\\
%     u(t)
%    \end{bmatrix}+\begin{bmatrix}
%      0\\
%      0\\
%      k_2
%    \end{bmatrix}\mu_*(t).
%  \end{equation}
%  Now let us consider that the control input or the measured output is delayed by a constant delay $h>0$. The closed-loop system can be rewritten as
\blue{\begin{equation}\label{eq:delay}
\begin{array}{rcl}
    \begin{bmatrix}
      \dot{x}_1(t)\\
      \dot{x}_2(t)\\
      \dot{u}(t)
    \end{bmatrix}&=&\begin{bmatrix}
      -\gamma_r & 0 & 0\\
      k_p & -\gamma_p & 0\\
       -k_1k_p & k_1\gamma_p-k_2 & 0
    \end{bmatrix}\begin{bmatrix}
     x_1(t)\\
     x_2(t)\\
     u(t)
    \end{bmatrix}\\&&\hfill+\begin{bmatrix}
      0 & 0 & 1\\
      0 & 0 & 0\\
      0 & 0 & 0
    \end{bmatrix}\begin{bmatrix}
     x_1(t-h)\\
     x_2(t-h)\\
     u(t-h)
    \end{bmatrix}+\begin{bmatrix}
      0\\
      0\\
       k_2
    \end{bmatrix}
      \mu_*.\hfill
    \end{array}
  \end{equation}
 We already know that the system can be made stable when $h=0$. Intuitively, the delay will deteriorate the stability of the system since it will be controlled with past information instead of current one. This leads to the following result:
   \begin{proposition}
   Assume that the gains $k_1$ and $k_2$ are given and satisfy the conditions of Lemma \ref{lem:jdksjdlks}. Then, the system \eqref{eq:delay} is asymptotically stable for all $h\in[0,h_c)$ and unstable otherwise where
   \begin{equation}\label{eq:hc}
      h_c:=-\dfrac{1}{\omega_c}\arg\left(\dfrac{-j\omega_c(j\omega_c+\gamma_r)(j\omega_c+\gamma_p)}{k_p(k_1j\omega_c+k_2)}\right)
   \end{equation}
   where $\omega_c$ is the only positive solution to the algebraic equation
      $$\Phi(\omega):=
      \omega^6+(\gamma_p^2+\gamma_r^2)\omega^4 +(\gamma_r^2\gamma_p^2-k_p^2k_1^2)\omega^2-k_p^2k_2^2=0.\mendth$$
   %\end{equation}
   \end{proposition}
 \begin{proof}
This result can be proven using a standard root analysis of the characteristic equation. To this aim, define the characteristic equation of  the closed-loop system as $K(s,h):=P(s)+e^{-sh}Q(s)$ where  $P(s):=s(s+\gamma_r)(s+\gamma_p)$ and $Q(s):=k_p(k_1s+k_2)$.
  %\begin{equation}
%    D(s,h):=P(s)+e^{-sh}Q(s)
%  \end{equation}
%  where
%  \begin{equation}
%      P(s):=s(s+\gamma_r)(s+\gamma_p)\ \textnormal{and } Q(s):=k_p(k_1s+k_2).
%  \end{equation}
  %
  There exists a pair of complex conjugate root on the imaginary axis for $K(s,h)$ iff $D(j\omega,h)=0$ for some $\omega>0$. This is equivalent to saying that $|P(j\omega)|^2-|Q(j\omega)|^2=0$. Calculations show that the equality $|P(j\omega)|^2-|Q(j\omega)|^2$ coincides with $\Phi(\omega)$. To study the existence of positive solutions to the algebraic equation $\Phi(\omega)=0$, we can alternatively consider the third-order algebraic equation $\Phi(\omega^{1/2})=0$. Invoking then Descartes' rule of sign, the number of sign changes in the coefficients is always 1, meaning that both $\Phi(\omega)=0$ and $\Phi(\omega^{1/2})=0$ have one and only one positive solution, which we denoted by $\omega_c$. Let us then define $(\omega_c,h_c)$ to be the pair for which $K(j\omega_c,h_c)=0$. Then, we get that  $e^{-j\omega_ch_c}=-P(j\omega_c)/Q(j\omega_c)$ and hence $h_c$ is given by \eqref{eq:hc}.
 \end{proof}}

\blue{\subsection{Example}

For simulation purposes, we consider the normalized version of system \eqref{eq:mainsyst} taken from \cite{Khammash:11} and given by:
\begin{equation}\label{eq:normalized}
  \begin{array}{lcl}
    \dot{\bar{x}}_1(t)&=&-(\gamma_r^0+u_2)\bar{x}_1(t)+\tilde{u}_1(t)\\
    \dot{\bar{x}}_2(t)&=&\gamma_p(\bar{x}_1(t)-\bar{x}_2(t))\\
    \dot{\bar{x}}_3(t)&=&(\gamma_r^0+u_2(t))\bar{x}_1(t)-2(\gamma_r^0+u_2)\bar{x}_3(t)+\tilde{u}_1(t)\\
    \dot{\bar{x}}_4(t)&=&(\gamma_r^0+\gamma_p)\bar{x}_3(t)-(\gamma_r^0+u_2+\gamma_p)\bar{x}_4(t)\\
    \dot{\bar{x}}_5(t)&=&\dfrac{\gamma_p}{\alpha}\left[\bar{x}_1(t)+\bar{x}_2(t)+2(\alpha-1)\bar{x}_4(t)\right]-2\gamma_p\bar{x}_5(t)
  \end{array}
\end{equation}
where $\tilde{u}_1(t)=\gamma_r^0+bu_1(t)$, $\gamma_r^0=0.03$, $\gamma_p=0.0066$, ${b=0.9587}$, $k_p=0.06$ and $\alpha=1+k_p/(\gamma_r^0+\gamma_p)$. The system has been normalized according to basal levels for transcription rate $k_r^0$ and degradation rate $\gamma_r^0$. In the absence of control inputs, i.e. $\tilde{u}_1\equiv0$ and $u_2\equiv0$, the system converges to the normalized equilibrium values $\bar{x}_i^*=1$, $i=1,\ldots,5$. The parameter values are borrowed from \cite{Khammash:11}. For this system, the global stability conditions write
\begin{equation}
  \dfrac{\gamma_r^0(\mu-1)}{b}\le\bar{u}, k_1>\dfrac{k_2}{\gamma_p+\gamma_r^0},k_2>0.
\end{equation}
We choose $\bar{u}=\infty$, $k_1=0.1191$ and $k_2=0.0007$, which satisfy the above conditions.

Simulations yield the trajectories of Fig. \ref{fig:mean1} and Fig. \ref{fig:mean2}. We can see, as expected, that the proposed controller achieves output tracking for different references and in presence of a saturation on the input, although convergence takes longer in the latter case.}
\begin{figure}[h]
\centering
\includegraphics[width=0.8\textwidth]{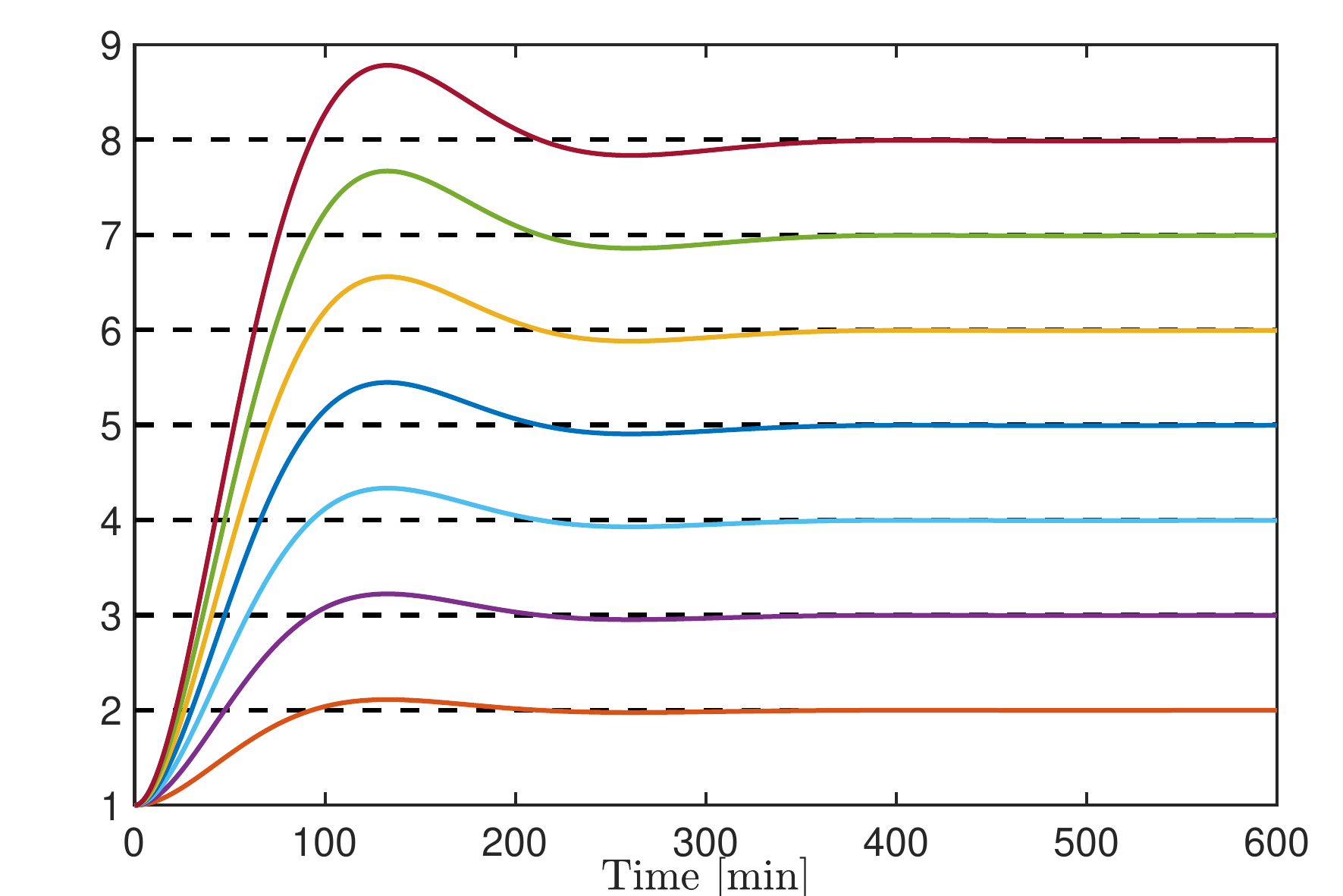}
\caption{Trajectories of the mean number of proteins for different reference values.}\label{fig:mean1}
\end{figure}
\begin{figure}[h]
\centering
\includegraphics[width=0.8\textwidth]{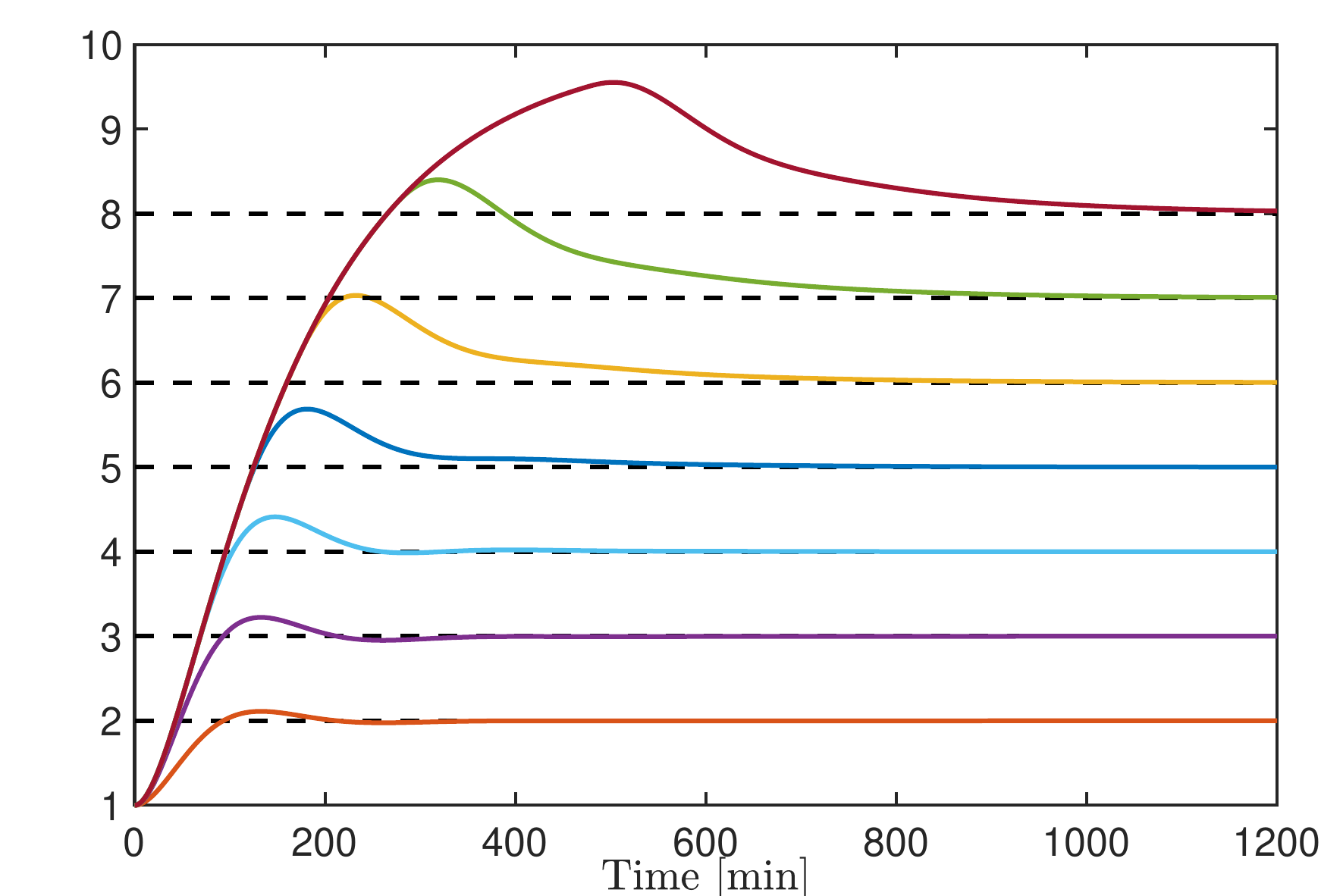}
\caption{Trajectories of the mean number of proteins for different reference values and with a saturation of $\bar{u}=0.3$.}\label{fig:mean2}
\end{figure}
%
%%
%\begin{minipage}{\linewidth}
%\hfill
%\begin{minipage}[t]{0.45\textwidth}
%\begin{figure}[H]
%\centering
%\includegraphics[width=0.8\textwidth]{./MatlabCDC12/mean_id.pdf}
%\caption{Trajectories of the mean number of proteins in response to a constant input disturbance.}\label{fig:mean2}
%\end{figure}
%\end{minipage}
%\hfill
%\begin{minipage}[t]{0.45\textwidth}
%\begin{figure}[H]
%\centering
%\includegraphics[width=0.8\textwidth]{./MatlabCDC12/mean_od.pdf}
%\caption{Trajectories of the mean number of proteins in response to a constant output disturbance.}\label{fig:mean3}
%\end{figure}
%\end{minipage}
%\hfill
%\end{minipage}

\subsection{Concluding remarks}

\blue{The protein variance at equilibrium is given by
\begin{equation}
  \sigma_*^2=\mu_*\left(1+\dfrac{k_p}{\gamma_p+\gamma_r}\right)
\end{equation}
which shows that the equilibrium variance depends linearly on the mean. Therefore, it cannot be independently assigned to a desired value. We can also observe that a higher mean yields a higher variance, which motivates the aim of controlling the variance in order to keep it at a desired level.}

\section{Mean and Variance control of protein levels in a gene expression network}\label{sec:var}

As discussed in the previous section, acting on $k_r$ is not sufficient for controlling both the mean and variance equilibrium values. It is shown in this section that variance control can be achieved by adding the second control input $\gamma_r\equiv u_2$. Fundamental limitations of the control system are discussed first, then local stabilizability is addressed.

\subsection{Fundamental limitations}

Let us consider in this section the control inputs $k_r\equiv u_1$ and $\gamma_r\equiv u_2$. It is shown below that there is a fundamental limitation on the references values for the mean and variance.
\begin{proposition}
  The set of admissible reference values $(\mu_*,\sigma_*^2)$ is given by the open and nonempty set
  \begin{equation}
    \mathcal{A}:=\left\{(x,y)\in\mathbb{R}_{>0}^2:\ x<y< \left(1+\dfrac{k_p}{\gamma_p}\right)x\right\}
  \end{equation}
  where $k_p,\gamma_p>0$. \mendth
\end{proposition}
\begin{proof}
  The lower bound is imposed by the coefficient of variation which gives
  \begin{equation}
  \sigma_*^2=\left(1+\dfrac{k_p}{\gamma_r+\gamma_p}\right)\mu_*>\mu_*.
\end{equation}
The upper bound is imposed by the positivity of the unique equilibrium control inputs values given by
\begin{equation}
    u_1^*=\dfrac{\gamma_p}{k_p}\mu_*u_2^*\textnormal{ and }u_2^*=-\gamma_p+\dfrac{k_p\mu_*}{\sigma^2_*-\mu_*}
\end{equation}
which are well-posed since $\sigma^2_*-\mu_*>0$ according to the coefficient of variation constraint. The second equilibrium control input value $u_2^*$ is positive if and only if $\textstyle{\sigma_*^2< \left(1+\frac{k_p}{\gamma_p}\right)\mu_*}$, which in turn implies that $u_1^*$ is nonnegative as well. The proof is complete.
\end{proof}
The lower bound obtained above remains valid when $k_p$ or $\gamma_p$ are chosen as second control inputs. The factor of the upper-bound however changes to $1+k_p/\gamma_r$ when $\gamma_p\equiv u_2$, or becomes unconstrained when $k_p\equiv u_2$. Note however that the upper bound on the variance is not a strong limitation in itself because we are mostly interested in achieving low variance.

Note also that since the lower bound on the achievable variance is independent of the controller structure, it is hence pointless to look for advanced control techniques in view of improving this limit. A positive fact, however, is that the lower bound is fixed and does not depend on the knowledge of the parameters of the system. This potentially makes low equilibrium variance robustly achievable.

\subsection{Problem formulation}

Considering the control inputs $k_r\equiv u_1$ and $\gamma_r\equiv u_2$, the system (\ref{eq:mainsyst}) can be rewritten as the bilinear system
\begin{equation}\label{eq:bilsyst}
  \begin{array}{lcl}
    \dot{x}_1&=&-u_2x_1+u_1\\
    \dot{x}_2&=&k_px_1-\gamma_px_2\\
    \dot{x}_3&=&u_2x_1-2u_2x_3+u_1\\
    \dot{x}_4&=&k_px_3-\gamma_px_4-u_2x_4\\
    \dot{x}_5&=&k_px_1+\gamma_px_2+2k_px_4-2\gamma_px_5\\
    \dot{I}_1&=&\mu_*-x_2\\
    \dot{I}_2&=&\sigma^2_*-x_5\\
  \end{array}
\end{equation}
where $I_1$ and $I_2$ are the states of the integrators. The control inputs are defined as the outputs of a multivariable positive PI controller
\blue{\begin{equation}\label{eq:clb}
  \begin{array}{lcl}
    u_1&=&\varphi_1\left(k_1e_1+k_2I_1+k_3e_2+k_4I_2\right)\\
    u_2&=&\varphi_2\left(k_5e_1+k_6I_1+k_7e_2+k_8I_2\right)
  \end{array}
\end{equation}}
where $e_1:=\mu_*-x_2$, $e_2:=\sigma_*^2-x_5$, \blue{and $\varphi_i(\cdot):=\min\{\max\{0,\cdot\},\bar{u}_i\}$, $\bar{u}_i>0$, $i=1,2$.}

\begin{property}
Assume that $k_2k_8-k_4k_6\ne0$, then the equilibrium point of the system (\ref{eq:bilsyst})-(\ref{eq:clb}) is unique and given by
\begin{equation}\label{eq:eqptv1}
  \begin{array}{l}
    x_1^*=\dfrac{\gamma_p}{k_p}\mu_*,\quad x_2^*=\mu_*,\quad  x_3^*=x_1^*,\quad x_4^*=\dfrac{\gamma_p}{\gamma_p+u_2^*}\mu_*,\vspace{1mm}\\
    x_5^*=\sigma_*^2,\quad u_1^*=\dfrac{\gamma_p}{k_p}\mu_*u_2^*,\quad  u_2^*=-\gamma_p+\dfrac{k_p\mu_*}{\sigma^2_*-\mu_*}
  \end{array}
\end{equation}
and
\begin{equation}\label{eq:eqptv2}
  \left[\begin{array}{c}
    I_1^*\\
    I_2^*
  \end{array}\right]=\left[\begin{array}{cc}
    k_2 & k_4\\
    k_6 & k_8
  \end{array}\right]^{-1}\left[\begin{array}{c}
    u_1^*\\
    u_2^*
  \end{array}\right]
\end{equation}
\blue{provided that $u^*_i\le\bar{u}_i$, $i=1,2$.}
\end{property}
Associated with the set of admissible references $\mathcal{A}$, we define the set of equilibrium points as
\begin{equation*}
  \mathcal{X}^*:=\left\{(x^*,I^*)\in\mathbb{R}^7:\ (y_*,\sigma_*^2)\in\mathcal{A}, \blue{u_i^*\in[0,\bar{u}_i],i=1,2}\right\}.
\end{equation*}

\subsection{Local stabilizability and stabilization}

Since the equilibrium control inputs are positive, the nonlinearities are not active in a neighborhood of the equilibrium point (\ref{eq:eqptv1})-(\ref{eq:eqptv2}). Local analysis can hence be performed using standard linearization techniques. The corresponding  Jacobian system is given by
\begin{equation}
  \dot{x}_\ell=A^*_\ell x_\ell
\end{equation}
where $A^*_\ell$ is given by
  \begin{equation}\label{eq:As}
\begin{bmatrix}
    \gamma_p+\frac{k_p\mu_*}{\delta} & -k_1+\frac{\gamma_pk_5\mu_*}{k_p} & 0 & 0 & -k_3+\frac{\gamma_pk_7\mu_*}{k_p} & k_2-\frac{\gamma_pk_6\mu_*}{k_p} & k_4-\frac{\gamma_pk_8\mu_*}{k_p}\\
    k_p & -\gamma_p & 0 & 0 & 0 & 0 & 0\\
    -\gamma_p-\frac{k_p\mu_*}{\delta} &  -k_1+\frac{\gamma_pk_5\mu_*}{k_p} & 2\gamma_p+2\frac{k_p\mu_*}{\delta} & 0 & -k_3+\frac{\gamma_pk_7\mu_*}{k_p} & k_2-\frac{\gamma_pk_6\mu_*}{k_p} & k_4-\frac{\gamma_pk_8\mu_*}{k_p}\\
    0 & -k_5\frac{\gamma_p\delta}{k_p} & k_p & \frac{k_p\mu_*}{\delta} & -\frac{\gamma_pk_7\delta}{k_p} & \frac{\gamma_pk_6\delta}{k_p} & \frac{\gamma_pk_8\delta}{k_p}\\
    k_p & \gamma_p & 0 & 2k_p & -2\gamma_p & 0 & 0\\
    0 & -1 & 0 & 0 & 0& 0 & 0\\
    0 & 0 & 0 & 0 & -1& 0 & 0
  \end{bmatrix}
\end{equation}
with $\delta:=\mu_*-\sigma_*^2$. The following result states conditions for the Jacobian system to be locally representative of the behavior of the original nonlinear system:
\begin{lemma}
 The Jacobian system fully characterizes the local behavior of the controlled nonlinear system (\ref{eq:bilsyst})-(\ref{eq:clb}) if and only if the condition $k_2k_8-k_4k_6\ne0$ holds.\mendth
\end{lemma}
\begin{proof}
%The condition comes from the fact that the determinant of the Jacobian must be nonzero.
  For the Jacobian system to represent the local behavior, it is necessary and sufficient that $A^*_\ell$ has no eigenvalue at 0. A quick check at the determinant value
  \begin{equation*}
    \det(A^*_\ell)=4\gamma_pk_p(k_2k_8-k_4k_6)(\mu_*(k_p+\gamma_p)-\gamma_p\sigma_*^2)
  \end{equation*}
  yields that the condition $k_2k_8-k_4k_6\ne0$ is necessary and sufficient for the local representativity of the nonlinear system. Note that since $(\mu_*,\sigma_*^2)\in\mathcal{A}$, the term ${\mu_*(k_p+\gamma_p)-\gamma_p\sigma_*^2}$ is always different from 0. The proof is complete.
  \end{proof}

The local system being linear, the Routh-Hurwitz criterion could have indeed been applied as in the mean control case, but would have led to quite complex algebraic inequalities, difficult to analyze in the general case, even for simple controller structures.  Despite the `large size' of the matrix $A^*_\ell$, it is fortunately still possible to provide a stabilizability result using the fact that $A^*_\ell$ is marginally stable when the controller parameters $k_i$ are set to 0. This is obtained using perturbation theory of nonsymmetric matrices \cite{Seyranian:03}.
\begin{lemma}
  Given any $k_p,\gamma_p>0$, the bilinear system (\ref{eq:bilsyst}) is locally asymptotically stabilizable around any equilibrium point (\ref{eq:eqptv1})-(\ref{eq:eqptv2}) using the control law (\ref{eq:clb}) \blue{provided that $(\mu_*,\sigma_*^2)\in\mathcal{A}$ and $u^*_i\le\bar{u}_i$, $i=1,2$.}  \mendth
\end{lemma}
\begin{proof}
  The perturbation argument \cite[Theorem 2.7]{Seyranian:03} relies on checking whether the eigenvalues on the imaginary axis can be shifted by slightly perturbing the controller coefficients around the `0-controller', i.e. by letting $k_i=\eps d_i$, where $\eps\ge0$ is the small perturbation parameter and $d_i$ is the perturbation direction corresponding to the controller parameter $k_i$. We assume here that both integrators are involved in the controller, that is ${|d_2|+|d_6|>0}$ and ${|d_4|+|d_8|>0}$. To prove the result, let us first rewrite the matrix $A^*_\ell$ as $ A^*_\ell=A_0+\eps\sum_{j=1}^8d_jA_j$.
%
%\begin{equation}
%  A^*_\ell=A_0+\eps\sum_{j=1}^8d_jA_j.
%\end{equation}
%
The matrix $A_0$ is a marginally stable matrix with a semisimple eigenvalue of multiplicity two at zero. Paradoxically, these eigenvalues introduced by the PI controller are the only critical ones that must be stabilized, i.e. shifted to the open left-half plane. From perturbation theory of general matrices \cite{Seyranian:03}, it is known that semisimple eigenvalues bifurcate into (distinct or not) eigenvalues according to the expression \cite{Seyranian:03} $\lambda_i(\eps,d)=\eps\xi_i(d)+o(\eps)$, $i=1,2$
%\begin{equation}\label{eq:eigexp}
%  \lambda_i(\eps,d)=\eps\xi_i(d)+o(\eps),\ i=1,2
%\end{equation}
where $\xi_i(d)$ is the $i^{th}$ eigenvalue of the matrix ${M(d):=\sum_{j=1}^8d_j\nu_\ell A_j\nu_r}$ where
%\begin{equation}
%M_j:=\begin{bmatrix}
%    \nu_\ell^1\\
%    \nu_\ell^2
%  \end{bmatrix}A_j\begin{bmatrix}
%   \nu_r^1&& \nu_r^1
%  \end{bmatrix},\ i=1,\ldots,8.
%\end{equation}
%$M_i:=\begin{bmatrix}
%    \nu_\ell^1\\
%    \nu_\ell^2
%  \end{bmatrix}A_j\begin{bmatrix}
%   \nu_r^1&& \nu_r^1
%  \end{bmatrix},\ i=1,\ldots,8.$
  $\nu_\ell\in\mathbb{R}^{2\times  5}$ and $\nu_r\in\mathbb{R}^{5\times  2}$ denote the normalized left- and right-eigenvectors\footnote{Normalized eigenvectors verify the conditions $\nu_\ell\nu_r=I$.} associated with the semisimple zero eigenvalue. It turns out that all $\nu_\ell A_j\nu_r$'s with odd index are zero, indicating that the proportional gains have a locally negligible stabilizing effect. This hence reduces the size of the problem to 4 parameters, i.e. those related to integral terms. We make now the additional restriction that $d_4=d_6=0$ reducing the controller structure to one integrator per control channel. The matrix $M(d)$ then becomes
\begin{equation*}
  M(d)=\psi\begin{bmatrix}
        \frac{k_pd_2}{\gamma_p} & -d_8\mu_*\\
        \frac{k_p\sigma_*^2d_2}{\gamma_p\mu_*} & \frac{\gamma_p(\mu_*-\sigma_*^2)^2}{k_p\mu_*}+\mu_*-2\sigma_*^2
  \end{bmatrix}
\end{equation*}
where $\psi:=\frac{\mu_*-\sigma_*^2}{\gamma_p(\mu_*-\sigma_*^2)+k_p\mu_*}$. The zero semisimple eigenvalues then move to the open left-half plane if there exist perturbation directions $d_2,d_8\in\mathbb{R}$, $d_2d_8\ne0$, such that $M(d)$ is Hurwitz stable. We can now invoke the Routh-Hurwitz criterion on $M(d)$ and we get the conditions
\begin{equation*}
  \begin{array}{rcl}
    d_2d_8\psi\frac{(\mu_*-\sigma_*^2)^2}{\gamma_p\mu_*}&>&0\\
    \gamma_p\psi\left(d_8\left(\gamma_p(\mu_*-\sigma_*^2)^2-2k_p\mu_*\sigma_*^2\right)+d_2k_p\mu_*\sigma_*^2\right)&<&0.
  \end{array}
\end{equation*}
Since the term $\psi$ is negative for all $(\mu_*,\sigma_*^2)\in\mathcal{A}$, the first inequality holds true if and only if $d_2d_8<0$, i.e. perturbation directions have different signs. The second inequality can be rewritten as
\begin{equation}\label{eq:lastcond}
\textstyle  d_2>d_8\left(2-\frac{\gamma_p(\mu_*-\sigma_*^2)^2}{k_p\mu_*\sigma_*^2}\right).
\end{equation}
Choosing then $d_8<0$, there always exists $d_2>0$ such that the above inequality is satisfied, making thus the matrix $M(d)$ Hurwitz stable. We have hence proved that for any given pair $(\mu_*,\sigma_*^2)\in\mathcal{A}$, there exists a control law (\ref{eq:clb}) that makes the corresponding equilibrium locally asymptotically stable. In fact, it is possible to show that there exists a pair $(d_2,d_8)\in\mathbb{R}^2$, $d_2d_8\ne0$ such that for all $(\mu_*,\sigma_*^2)\in\mathcal{A}$ the inequality (\ref{eq:lastcond}) is satisfied. This is equivalent to finding a finite $d_2>0$ satisfying
\begin{equation}
d_2>d_8\left(2-\sup_{(\mu_*,\sigma_*^2)\in\mathcal{A}}\left\{\frac{\gamma_p(\mu_*-\sigma_*^2)^2}{k_p\mu_*\sigma_*^2}\right\}\right).
\end{equation}
Standard analysis allows us to prove that
\begin{equation} \sup_{(\mu_*,\sigma_*^2)\in\mathcal{A}}\left\{\frac{\gamma_p(\mu_*-\sigma_*^2)^2}{k_p\mu_*\sigma_*^2}\right\}=\dfrac{k_p}{\gamma_p+k_p}\in(0,1)
\end{equation}
which shows that by simply choosing the directions $d_8<0$ and $d_2>0$, the matrix $M(d)$ hence becomes Hurwitz stable for all $(\mu_*,\sigma_*^2)\in\mathcal{A}$.
\end{proof}

\blue{\subsection{Example}\label{sec:ex2}

We consider here a PI controller with gains $k_1=1$, $k_2=0.007$, $k_7=-0.2$ and $k_8=-0.0014$ (all the other gains are set to 0). %The normalized achievable minimal variance is given by
%\begin{equation}
%  \sigma_{min}^2=\dfrac{\gamma_r^0+\gamma_p}{\gamma_r^0+\gamma_p+k_p}\mu.
%\end{equation}
The response of the controlled mean and variance according to changes in their reference value is depicted in Fig. \ref{fig:var1} where we can see that both the mean and the variance converge to their respective set-points.  It seems also important to point out that when the reference point changes, due to the coupling between the mean and variance, the mean value changes as well, but this is immediately corrected by the mean controller. For all those set-points, it can be verified that the linearized system is locally exponentially stable.}%The input disturbance rejection properties are depicted in Figure \ref{fig:var4} and Figure \ref{fig:var5}}

\begin{figure}[h]
\centering
\includegraphics[width=0.8\textwidth]{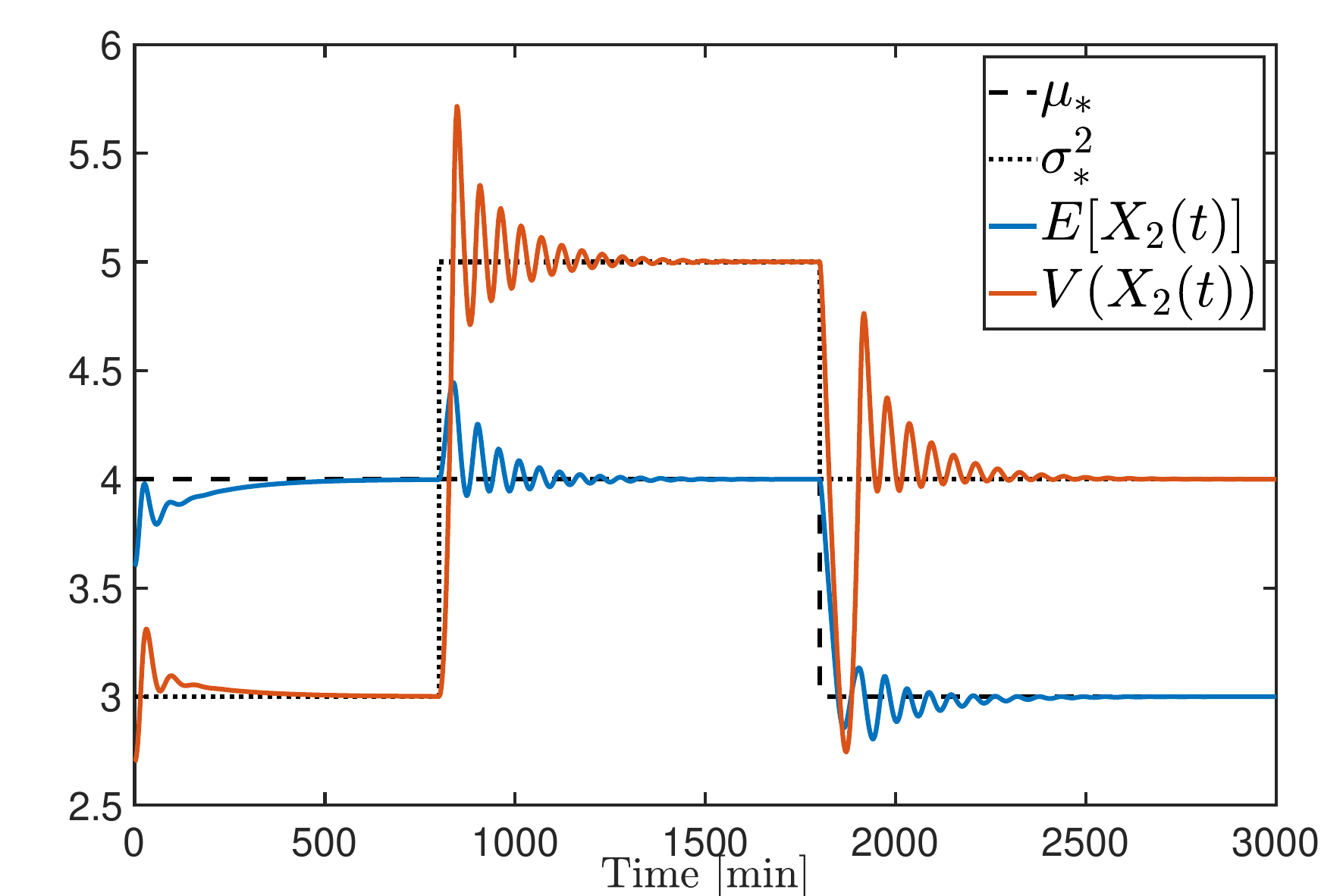}
\caption{Trajectories of the protein mean and variance subject to changes in their references.}\label{fig:var1}
\end{figure}

%
%\begin{figure}[h]
%\centering
%\includegraphics[width=0.3\textwidth]{./MatlabCDC12/var_ot2.pdf}
%\caption{Response of the controlled variance according to changes in the reference.}\label{fig:var1}
%\end{figure}
%
%
%\begin{minipage}{\linewidth}
%\hfill
%\begin{minipage}[t]{0.45\textwidth}
%\begin{figure}[H]
%\centering
%\includegraphics[width=0.8\textwidth]{./MatlabCDC12/var_od1.pdf}
%\caption{Response of the controlled variance subject to a constant output disturbance on $y_1$}\label{fig:var2}
%\end{figure}
%\end{minipage}
%\hfill
%\begin{minipage}[t]{0.45\textwidth}
%\begin{figure}[H]
%\centering
%\includegraphics[width=0.8\textwidth]{./MatlabCDC12/var_od2.pdf}
%\caption{Response of the controlled variance subject to a constant output disturbance on $y_2$}\label{fig:var3}
%\end{figure}
%\end{minipage}
%\hfill
%\end{minipage}
%
%
%\begin{minipage}{\linewidth}
%\hfill
%\begin{minipage}[t]{0.45\textwidth}
%\begin{figure}[H]
%\centering
%\includegraphics[width=0.8\textwidth]{./MatlabCDC12/var_id1.pdf}
%\caption{Response of the controlled variance subject to a constant input disturbance on $u_1$}\label{fig:var4}
%\end{figure}
%\end{minipage}
%\hfill
%\begin{minipage}[t]{0.45\textwidth}
%\begin{figure}[H]
%\centering
%\includegraphics[width=0.8\textwidth]{./MatlabCDC12/var_id2.pdf}
%\caption{Response of the controlled variance subject to a constant input disturbance on $u_2$}\label{fig:var5}
%\end{figure}
%\end{minipage}
%\hfill
%\end{minipage}

%\newpage

\section{Mean control of the dimer in the gene expression network with dimerization}\label{sec:dimer}

\subsection{The model}

We consider the following model for the gene expression network with protein dimerization
\begin{equation}\label{eq:reacnet}
\begin{array}{lcccllcccl}
   R_1&:&\phib&\stackrel{k_1}{\longrightarrow}&\X{1},&  R_2&:&\X{1}+\X{1}&\stackrel{b}{\longrightarrow}&\X{2},\\
   R_3&:&\X{1}&\stackrel{\gamma_1}{\longrightarrow}&\phib,&  R_4&:&\X{2}&\stackrel{\gamma_2}{\longrightarrow}&\phib
 \end{array}
\end{equation}
in which the protein $\X{1}$ dimerizes into $\X{2}$ at rate $b$. %Note that the mRNA-stage has been skipped. The reason behind this simplification is that the mRNA-stage would simply add unimolecular reactions that would not change anything to the overall approach and its conclusions as all the difficulties arise from the presence of a bimolecular reaction. The stoichiometry matrix associated with the gene expression network is given by
%\begin{equation}
%    S=\begin{bmatrix}
%      1 & -2 & -1 & 0\\
%      0 & 1 & 0 & -1
%    \end{bmatrix}
%\end{equation}
%and, assuming mass-action kinetics, the vector of propensity functions is given by
%\begin{equation}
%    w(X)=\begin{bmatrix}
%      k_1 & \frac{b}{2}X_1(X_1-1) & \gamma_1X_1 & \gamma_2X_2
%    \end{bmatrix}^T.
%\end{equation}
%%
For this network, the moments equation writes
\begin{equation}\label{eq:syst}
  \begin{array}{lcl}
    \dot{x}_1(t)&=&k_1+(b-\gamma_1)x_1(t)-bx_1(t)^2-bv(t)\\
    \dot{x}_2(t)&=&-\dfrac{b}{2}x_1(t)-\gamma_2 x_2(t)+\dfrac{b}{2}x_1(t)^2+\dfrac{b}{2}v(t)
  \end{array}
\end{equation}
where $x_i(t):=\E[X_i(t)]$, $i=1,2$, and $v(t):=V(X_1(t))$ is the variance of the random variable $X_1(t)$. We can immediately observe that the set of equations is not closed because of the presence of the variance acting as a \blue{disturbance} to the system. Another fundamental difference with the moments equation in the unimolecular case is that the above one is nonlinear, which adds a layer of complexity to the problem.

\subsection{Main difficulties}\label{sec:diff}

In spite of being simple, the  network \eqref{eq:reacnet} presents all the difficulties that can arise in bimolecular reaction networks and is a good candidate for emphasizing that moment control problem may remain solvable when the moments equations are not closed. The first issue is that the  system \eqref{eq:syst} has the variance $v(t):=V(X_1(t))$ as \blue{disturbance} signal and it is not known, a priori, whether it is bounded over time or even asymptotically converging to a finite value $v^*$. The second one is that the system \eqref{eq:syst} is nonlinear and nonlinear terms can not be neglected as they may enhance certain properties such as stability. It will be shown later that this is actually the case for system \eqref{eq:syst}. Finally, the last one arises from the fact that, due to our complete ignorance in the value of $v^*$ (if it exists), the system \eqref{eq:syst} exhibits an infinite number of equilibrium points. Understand this, however, as an artefact arising from the definition of the model \eqref{eq:syst} since the first-order moments may, in fact, have a unique stationary value.
%\end{enumerate}

\subsection{Preliminary results}

The following result proves a crucial stability property for our process:
\begin{theorem}\label{th:ergodicd}
  For any value of the network parameters $k_1,b,\gamma_1$ and $\gamma_2$, the reaction network \eqref{eq:reacnet} is exponentially ergodic and has all its moments bounded and globally exponentially converging. As a result, there exists a unique $v^*\ge0$ such that for any initial condition $X_0$, we have that $v(t)\to v^*$ as $t\to\infty$.\mendth
  %
  %Moreover, given any initial distribution $P(\varkappa,0):=\delta_{X_0}(\varkappa)$, the probability distribution $P(\varkappa,t)$ is light-tailed.
\end{theorem}
\begin{proof}
The proof is based on Proposition 10 in \cite{Briat:13i}. Indeed, the Foster-Lyapunov function $V(x)=\begin{bmatrix}
  1 & 2
\end{bmatrix}x$ satisfies the conditions in Proposition 10 and we have that
 \begin{equation}\label{eq:AV}
    \mathbb{A}V(x)   \le  c_1-c_2V(x)
  \end{equation}
  where  $c_1=k_1$ and $c_2=\min\{\gamma_1,\gamma_2\}$.
\end{proof}

The above result provides an answer to the first difficulty mentioned in Section \ref{sec:diff}. It indeed states that, for any parameter configuration, all the moments are bounded and exponentially converging to a unique stationary value.

\blue{The next step consists of choosing a suitable control input, that is, a control input from which any reference value $\mu_*$ for $x_2$ can be tracked. We propose to use the production rate $k_1$ as control input, a choice justified by the following result which states that all the stationary moments are continuous functions of the input:
\begin{proposition}\label{ass:2}
 Let $(i,j)\in\mathbb{Z}_{\ge0}^2$. Then, for the reaction network  \eqref{eq:reacnet}, the map
  \begin{equation}
    m_{i,j}:k_1\mapsto \E_\pi[X_1^iX_2^j]
  \end{equation}
is continuous where $\E_\pi[\cdot]$ denotes the expectation at stationarity.
\end{proposition}
\begin{proof}
  We actually prove the stronger result that the map $ m_{i,j}$ is differentiable for any $(i,j)\in\mathbb{Z}_{\ge0}^2$. To show that, we use \cite[Theorem 3.2]{Gupta:14b} in the particular case described in Section 3.1 of the same reference. All the conditions (which are too numerous to be all recalled here) of this theorem are trivially satisfied for the considered reaction network and the functions $m_{i,j}$'s, which all remain bounded on the state-space by virtue of Theorem \ref{th:ergodicd}. As a result, the differentiability of the map follows.
\end{proof}}
% we need the following assumption motivated by the structure of the network \eqref{eq:reacnet}:
%%
%\begin{hyp}\label{ass:2}
%The maps $k_1\mapsto \E_\pi[X]$ and $k_1\mapsto \E_\pi[X_1^2]$ are continuous functions.
%
%The function $S^*:=x_1^{*2}-x_1^*+v^*$, where $x_1^*$ is the equilibrium solution for $x_1$ and $v^*$ is the equilibrium variance, verifying the equation
%  \begin{equation}
%    k_1-\gamma x_1^*-bS^*=0,
%  \end{equation}
%is a continuous function of $k_1$.\mendth
%\end{hyp}

%By indeed increasing $k_1$, we will have more $X_1$ at stationarity, and consequently more $X_1^2$. It seems important to stress here that the continuity of the stationary distribution with respect to the network parameters cannot be assessed from the continuity of the probability distribution over time since the limit of continuous functions need not be continuous. Therefore, an argument based on the continuity of the stationary distribution seems difficult to consider.

We can now state the following result:
\begin{proposition}\label{prop:mu}
 For any $\mu_*>0$, there exists a $k_1=k_1(\mu_*)>0$ such that we have $x_2^*:=\E_\pi[X_2]=\mu_*$.\mendth
\end{proposition}
\begin{proof}
The question that has to be answered is whether for any $\mu_*$ the set of equations
\begin{equation}\label{eq:dksldksld}
  \begin{array}{rcl}
     k_1+(b-\gamma_1)x_1^*-bx_1^{*2}-bv^*&=&0\\
    -\dfrac{b}{2}x_1^*-\gamma_2\mu_*+\dfrac{b}{2}x_1^{*2}+\dfrac{b}{2}v^*&=&0
  \end{array}
\end{equation}
has a solution in terms of $k_1$ and $x_1^*$, where $x_1^*$ and $v^*$ are equilibrium values for $x_1$ and $v$. In the following, we define $S(t):=x_1(t)^{2}-x_1(t)+v(t)$ and let $S^*=S^*(k_1)=\E_\pi[X_1(X_1-1)]$ be its value at equilibrium that satisfies the first equation of the system \eqref{eq:dksldksld}. In this respect, the above equations can be rewritten as
\begin{equation}\label{eq:dskldjsldsld5678}
  \begin{array}{rcl}
     k_1-\gamma_1x_1^*-bS^*&=&0\\
     -\gamma_2\mu_*+\dfrac{b}{2}S^*&=&0.%\\
%    -\gamma_2\mu_*+\dfrac{b}{2}S^*&=&0.
  \end{array}
\end{equation}
Based on the above reformulation, we can clearly see that if we can set $S^*$ to any value by a suitable choice of $k_1$, then any $\mu_*$ can be achieved. We prove this in what follows.

\textbf{Step 1.} First of all, we have to show that when $k_1=0$, we have that $S^*=0$ and $x_1^*=0$. This is immediate from the fact that  $(0,0)$ is an absorbing state for the Markov process describing the network.
%This can be viewed directly from  \eqref{eq:AV} which implies that  $\lim_{t\to\infty}\E[X_1(t)+2X_2(t)]\le k_1/\min\{\gamma_1,\gamma_2\}$. Hence, if $k_1=0$, we have that $x_1^*=0$. To see that $S^*$ also zero, simply note that when $k_1=0$, then $(0,0)$ is an absorbing state for the Markov process
%
%the results of \cite{Briat:13i} which states that the asymptotic moment bounds for the first-order moment of $V(x)=\nu^T x$ is given by $c_1/c_2$, i.e. $\lim_{t\to\infty}\E[]\le c_1/c_2$, where $c_1,c_2$ are defined as in
%
%
%Therefore, $\lim_{t\to\infty}\E[X_1(t)]\le c_1/c_2$. This implies that when $k_1=0$, then $\E[X_1(t)]\to x_1^*=0$ as $t\to\infty$.

\textbf{Step 2.} We show now that when $k_1$ grows unbounded, then $S^*$ grows unbounded as well. To do so, let us focus on the first equation of \eqref{eq:dskldjsldsld5678}. Two options: either both $x_1^*$ and $S^*$ tend to infinity, or only one of them grows unbounded and the other remains bounded. We show that $S^*$ has to grow unbounded. Let us assume that $S^*=S^*(k_1)$ is uniformly bounded in $k_1$, i.e. there exists $\bar{S}>0$ such that $S^*\in[0,\bar{S}]$ for all $k_1\ge0$. Then, from the first equation of \eqref{eq:dskldjsldsld5678}, we have that $x_1^*=(k_1-bS^*)/\gamma_1$ and thus  $x_1^*\ge\bar{x}_1:=(k_1-b\bar{S})/\gamma_1$ for all $k_1\ge0$. Hence, $x_1^*$ grows unbounded as $k_1$ increases to infinity. From Jensen's inequality, we have that $S^*\ge x_1^*(x_1^*-1)$. Noting then that for the function $f(x):=x(x-1)$, we have that $f(y)\ge f(x)$ for all $y\ge x$, $x\ge1$, we can state that
\begin{equation}
  \bar{x}_1(\bar{x}_1-1)\le x_1^*(x_1^*-1)\le S^*\le\bar{S}
\end{equation}
for all $k_1>0$ such that $\bar{x}_1\ge1$. It is now clear that for any $\bar{S}>0$, there exists $k_1>0$ such that the above inequality is violated since $f(\bar{x}_1)$ can be made arbitrarily large. Therefore, $S^*$ must go to infinity as $k_1$ goes to infinity.

\blue{Using Proposition \ref{ass:2}, we have that the map $k_1\mapsto \E_\pi[X_1(X_1-1)]=S^*(k_1)$ is continuous. As a result, we can conclude that for any $\mu_*>0$, there will exist $k_1>0$, such that we have $x_2^*=\mu_*$.  The proof is complete.}
\end{proof}
\subsection{Nominal stabilization result}

From the results stated in Theorem \ref{th:ergodicd} and Proposition \ref{prop:mu}, it seems reasonable to consider a pure integral control law since exponential stability nominally holds and only tracking is necessary. Therefore, we propose that $k_1$ be actuated as
\begin{equation}\label{eq:cl}
\begin{array}{rcl}
  \dot{I}(t)&=&\mu_*-x_2(t)\\
    k_1(t)&=&k_c\varphi(I(t))
\end{array}
\end{equation}
where $k_c>0$ is the gain of the controller, $\mu_*$ is the reference to track and $\varphi(y):=\max\{0,y\}$. We are now in position to state the main result of the paper:
\begin{theorem}[Main stabilization result]\label{th:main}
  For any finite positive constants $\gamma_1,\gamma_2,b,\mu_*$ and any controller gain $k_c$ satisfying
  \begin{equation}\label{eq:kcbound}
  0<k_c<2\gamma_2\left(2\gamma_1+\gamma_2+2\sqrt{\gamma_1(\gamma_1+\gamma_2)}\right),
\end{equation}
  the closed-loop system \eqref{eq:syst}-\eqref{eq:cl} has a unique locally exponentially stable equilibrium point $(x_1^*,x_2^*,I^*)$ in the positive orthant such that $x_2^*=\mu_*$.
%  \begin{equation}
%    \lim_{t\to\infty}\E[X_2(t)]=\mu_*.
%  \end{equation}
The equilibrium variance moreover satisfies
\begin{equation*}
\begin{array}{lcr}
   \quad\quad\quad\quad\quad\quad&v^*\in\left(0,\dfrac{2\gamma_2\mu_*}{b}+\dfrac{1}{4}\right].& \quad\quad\quad\quad\quad\quad\vartriangle
\end{array}
\end{equation*}
\end{theorem}
\begin{proof}
\textbf{Location of the equilibrium points and local stability results.} We know from Proposition \ref{prop:mu} that for any $\mu_*>0$, the set of equations \eqref{eq:dksldksld}  has solutions in terms of the equilibrium values $x_1^*,x_2^*,I^*$ and $v^*$. Adding two times the second equation to the first one and multiplying the second one by $2/b$, we get that
\begin{equation}
  \begin{array}{rcl}
    k_cI^*-\gamma_1 x_1^*-\gamma_2\mu_*&=&0\\
    x_1^{*2}-x_1^*+v^*-\dfrac{2\gamma_2\mu_*}{b}&=&0.
  \end{array}
\end{equation}
The first equation immediately leads to $ I^*=(\gamma_1 x_1^*+\gamma_2\mu_*)/k_c$ which is positive for all $\mu_*>0$. This also means that $x_1^*$ is completely characterized by the equation
\begin{equation}
    x_1^{*2}-x_1^*+v^*-\dfrac{2\gamma_2\mu_*}{b}=0.
\end{equation}
The goal now is to determine the location of the solutions $x_1^*$ to the above equation where $v^*$ is viewed as an unknown parameter, reflecting our complete ignorance on the value $v^*$. We therefore embed the actual unique equilibrium point (from ergodicity and moments convergence) in a set having elements parametrized by $v^*\ge0$. The equation to be solved is quadratic, and it is a straightforward implication of Descartes' rule of signs \cite{Khovanskii:91} that we have three distinct cases:
\begin{itemize}
  \item[1)] If $v^*-2\gamma_2\mu_*/b<0$, then we have one positive equilibrium point.
  \item[2)] If $v^*-2\gamma_2\mu_*/b=0$, then we have one equilibrium point at zero, and one which is positive.
  \item[3)] If $v^*-2\gamma_2\mu_*/b>0$, then we have either 2 complex conjugate equilibrium points, or 2 positive equilibrium points.
\end{itemize}

\noindent  \textbf{Case 1:} This case holds whenever $v^*\in\left[0,\frac{2\gamma_2\mu_*}{b}\right)$
and the only positive solution for $x_1^*$ is given by $\textstyle{x_1^*=\frac{1}{2}\left(1+\sqrt{\Delta}\right)}$ where
\begin{equation}\label{eq:delta}
\Delta=1+4\left(\dfrac{2\gamma_2\mu_*}{b}-v^*\right)>1.
\end{equation}
The equilibrium point is therefore given by
\begin{equation}
  z^*=\left[\dfrac{1}{2}\left(1+\sqrt{\Delta}\right),\mu_*, \dfrac{\gamma_2\mu_*+\gamma_1x_1^*}{k_c}\right].
\end{equation}
The linearized system around this equilibrium point reads
\begin{equation}
  \dot{\tilde{z}}(t)=\begin{bmatrix}
    -\gamma_1-b\sqrt{\Delta} & 0 & k_c\\
    b\sqrt{\Delta}/2 & -\gamma_2 & 0\\
    0 & -1 & 0
  \end{bmatrix}\tilde{x}(t)+\begin{bmatrix}
      -b\\
      b/2\\
      0
    \end{bmatrix}\tilde{v}(t)
\end{equation}
where $\tilde{z}(t):=z(t)-z^*$, $z(t):=\col(x(t),I(t))$ and $\tilde{v}(t):=v(t)-v^*$. The Routh-Hurwitz criterion allows us to derive the stability condition
\begin{equation}\label{eq:condKc1}
  0<k_c<\dfrac{2\gamma_2(\gamma_1+\gamma_2+b\sqrt{\Delta})(\gamma_1+b\sqrt{\Delta})}{b\sqrt{\Delta}}.
\end{equation}
  \noindent  \textbf{Case 2:} In this case, we have $v^*=2\gamma_2\mu_*/b$ and is rather pathological but should be addressed for completeness. Let us consider first the equilibrium point $x_1^*=0$ giving $z^*=\begin{bmatrix}
    0 & \mu_* & \gamma_2\mu_*/k_c
  \end{bmatrix}$. The linearized dynamics of the system around this equilibrium point is given by
\begin{equation}
  \dot{\tilde{z}}(t)=\begin{bmatrix}
    b-\gamma_1 & 0 & k_c\\
    -b/2 & -\gamma_2 & 0\\
    0 & -1 & 0
  \end{bmatrix}\tilde{z}(t)+\begin{bmatrix}
      -b\\
      b/2\\
      0
    \end{bmatrix}\tilde{v}(t).
\end{equation}
Since the determinant of the system matrix is positive, this equilibrium point is unstable. Considering now the equilibrium point $x_1^*=1$ and, thus, $ z^*=\begin{bmatrix}
    1 & \mu_* & (\gamma_2\mu_*+\gamma_1)/k_c
  \end{bmatrix}$, we get the linearized system
\begin{equation}
   \dot{\tilde{z}}(t)=\begin{bmatrix}
    -b-\gamma_1 & 0 & k_c\\
    b/2 & -\gamma_2 & 0\\
    0 & -1 & 0
  \end{bmatrix}\tilde{z}(t)+\begin{bmatrix}
      -b\\
      b/2\\
      0
    \end{bmatrix}\tilde{v}(t).
\end{equation}
which is exponentially stable provided that
\begin{equation}\label{eq:condKc2}
  k_c<\dfrac{2\gamma_2(\gamma_1+\gamma_2+b)(\gamma_1+b)}{b}.
\end{equation}
  \noindent  \textbf{Case 3:} This case corresponds to when $v^*>2\gamma_2\mu_*/b$. However, this condition is not sufficient for having positive equilibrium points and we must add the constraint $v^*<2\gamma_2\mu_*/b+1/4$ in order to have a real solutions for $x_1^*$. When the above conditions hold, the equilibrium points are given by
\begin{equation}
  z_\pm^*=\begin{bmatrix}
    \dfrac{1}{2}\left(1\pm\sqrt{\Delta}\right) & \mu_* & \dfrac{\gamma_2\mu_*+\gamma_1x_1^*}{k_c}
  \end{bmatrix}
\end{equation}
where  $\Delta$ is defined in \eqref{eq:delta}. Similarly to as previously, the equilibrium point $z_-^*$ can be shown to be always unstable and the equilibrium point $z^*_+$ exponentially stable provided that
\begin{equation}\label{eq:condKc3}
  k_c<\dfrac{2\gamma_2(\gamma_1+\gamma_2+b\sqrt{\Delta})(\gamma_1+b\sqrt{\Delta})}{b\sqrt{\Delta}}.
\end{equation}
Note that if the discriminant $\Delta$ was equal to 0, the system would be unstable.

\noindent  \textbf{Bounds on the variance.} We have thus shown that to have a unique positive locally exponentially stable equilibrium point, we necessarily have an equilibrium variance $v^*$ within the interval $ v^*\in\left(0,2\gamma_2\mu_*/b+1/4\right].$ If $v^*$ is greater than the upper-bound of this interval, the system does not admit any real equilibrium points.

\noindent  \textbf{Uniform controller bound.} The last part concerns the derivation of a uniform condition on the gain of the controller $k_c$ such that all the positive equilibrium points that can be locally stable are stable. In order words, we want to unify the conditions \eqref{eq:condKc1}, \eqref{eq:condKc2} and \eqref{eq:condKc3} all together. Noting that these conditions can be condensed to $k_c<f(\sqrt{\Delta})$ where
\begin{equation}
  f(\zeta):=\dfrac{2\gamma_2(\gamma_1+\gamma_2+b\zeta)(\gamma_1+b\zeta)}{b\zeta}.
\end{equation}
Moreover, since $\mu_*>0$ can be arbitrarily large (and thus $\Delta$ may take any nonnegative value), the worst case bound for $k_c$ coincides with the minimum of the above function for $\zeta\ge0$. Standard calculations show that the minimizer is  given by $\textstyle \zeta^*=(\gamma_1(\gamma_1+\gamma_2))^{1/2}/b$ and the minimum $f^*:=f(\zeta^*)$ is therefore given by
\begin{equation}
  f^*=2\gamma_2\left(2\gamma_1+\gamma_2+2\sqrt{\gamma_1(\gamma_1+\gamma_2)}\right).
\end{equation}
The proof is complete.
\end{proof}

The above result states two important facts that must be emphasized. First of all, the condition on the controller gain is uniform over $\mu_*>0$ and $b>0$, and is therefore valid for any combination of these parameters. This also means that a single controller, which locally stabilizes all the possible equilibrium points, is easy to design for this network. Second, the proof of the theorem provides an explicit construction of an upper-bound on the equilibrium variance $v^*$, which turns out to be a linearly increasing function of $\mu_*$. This upper-bound is, moreover, tight when regarded as a condition on the equilibrium points of the system since, when the equilibrium variance $v^*$ is greater than $2\gamma_2\mu_*/b+1/4$, the system does not admit any real equilibrium point.

\subsection{Robust stabilization result}

Let us consider the following set
\begin{equation}
\mathcal{P}:=[\gamma_1^-,\gamma_1^+]\times[\gamma_2^-,\gamma_2^+]\times[b^-,b^+]
\end{equation}
defined for some appropriate positive real numbers $\gamma_1^-<\gamma_1^+$, $\gamma_2^-<\gamma_2^+$ and $b^-<b^+$. We get the following generalization of Theorem \ref{th:main}:
\begin{theorem}[Robust stabilization result]
Assume the controller gain $k_c$ verifies
  \begin{equation}
0<k_c<2\gamma_2^-\left(2\gamma_1^-+\gamma_2^-+2\sqrt{\gamma_1^-(\gamma_1^-+\gamma_2^-)}\right).
\end{equation}
Then, for all $(\gamma_1,\gamma_2,b)\in\mathcal{P}$, the closed-loop system \eqref{eq:syst}-\eqref{eq:cl} has a unique locally stable equilibrium point $(x_1^*,x_2^*,I^*)$ in the positive orthant such that $x_2^*=\mu_*$. The equilibrium variance $v^*$, moreover, satisfies
\begin{equation*}
\begin{array}{lcr}
  \quad\quad\quad\quad\quad\quad& v^*\in\left(0,\dfrac{2\gamma_2^+\mu_*}{b^-}+\dfrac{1}{4}\right]. & \quad\quad\quad\quad\quad\quad\vartriangle
\end{array}
\end{equation*}
\end{theorem}
\begin{proof}
  The upper bound on the controller gain is a strictly increasing function of $\gamma_1$ and $\gamma_2$, and the most constraining value (smallest) is therefore attained at $\gamma_1=\gamma_1^-$ and $\gamma_2=\gamma_2^-$. A similar argument is applied to the variance upper-bound.
\end{proof}

\subsection{Example}\label{sec:ex}

%\begin{algorithm*}
%  \caption{Algorithm for simulating the controlled cell population}
%
%  \begin{algorithmic}[1]
%   \Require{$T_s,\mu_*,k_c,T>0$, $N,N_p\in\mathbb{N}$, $\{x_0^1,\ldots,x_0^N\}\in\left(\mathbb{N}_0^2\right)^N$, $I_0\in\mathbb{R}$ and $p\in\mathbb{R}_{>0}^{N_p}$}
%   %
%   \State Create array $t$ of time instants from 0 to $T$ with time-step $T_s$.
%   \State $N_s=\text{length}(t)$
%   \State Initialize: $i\leftarrow1$, $y\leftarrow\text{mean}(x_0)$, $I\leftarrow I_0$
%    \For{$i<N_s$}
%    \State Update control input: $u\leftarrow k_c\cdot\max\{0,I\}$
%    \State Update controller state: $I\leftarrow I+T_s(\mu_*-y)$
%    \State Simulation of $N$ cells from time $t[i]$ to $t[i+1]$ with control input $u$ and network parameters $p$
%    \State Update output: $y\leftarrow\text{mean}(protein\ population)$
%    \State $i\leftarrow i+1$
%    \EndFor
%  \end{algorithmic}
%\end{algorithm*}

Let us consider in this section the stochastic reaction network \eqref{eq:reacnet} with parameters $b=3$, $\gamma_1=2$ and $\gamma_2=1$. From condition \eqref{eq:kcbound}, we get that $k_c$ must satisfy $  0<k_c< 19.798$ to have local stability of the unique equilibrium point in the positive orthant. We then run Algorithm 1 with the controller gain $k_c=1$, a sampling period of $T_s=10$ms, the reference $\mu_*=5$, the controller initial condition $I(0)=0$, a population of $N=10000$ cells and initial conditions $x_0^i$ randomized in $\{0,1\}^2$, $i=1,\ldots,N$. The simulation results are depicted in Fig. \ref{fig:average} to \ref{fig:variancedu}. We can clearly see in Fig.~\ref{fig:average} that $\E[X_2(t)]$ tracks the reference $\mu_*$ reasonably well. The variance of $X_1(t)$, plotted in Fig. \ref{fig:variance}, is also verified to lie within the theoretically determined range of values. We indeed have $V(X_1(t))\simeq1.5$ in the stationary regime whereas the upper-bound is equal to $3+7/12\simeq3.583$. Moreover, since the variance at equilibrium is smaller than $\frac{2\gamma_2\mu_*}{b}=10/3$, we then have $v-2\gamma_2\mu_*/b<0$ and therefore case 1) holds in the proof of Theorem \ref{th:main}. It is, however, unclear whether this is also the case for any combination of network parameters. We can see in Fig. \ref{fig:averagedu} that the apparition of a constant input disturbance of amplitude 15 at $t=15$ is efficiently taken care of by the controller.

\begin{figure}[H]
\centering
  \includegraphics[width=0.8\textwidth]{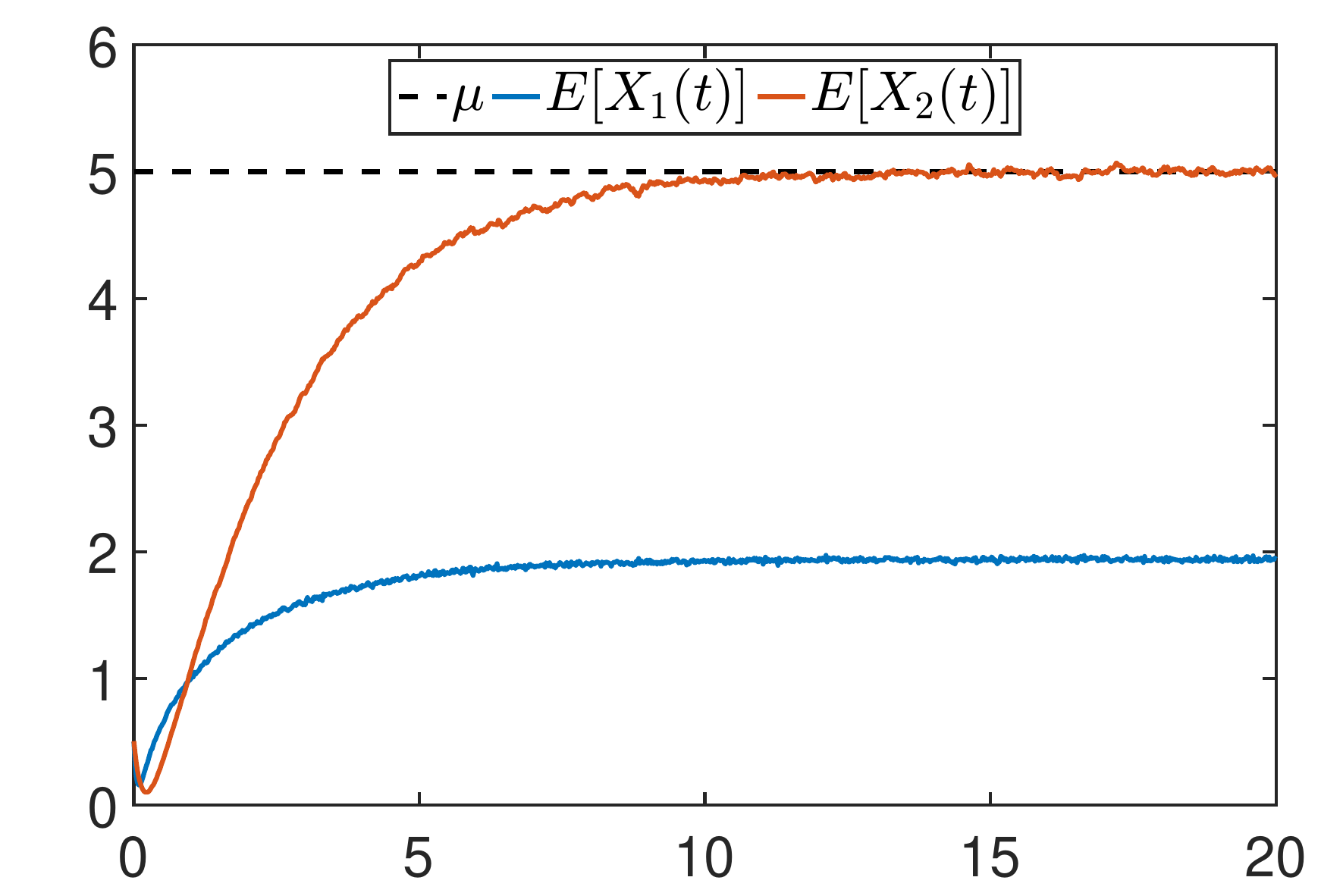}%\numfig}
  \caption{Evolution of the proteins averages in a population of 10000 cells (no disturbance)}\label{fig:average}
\end{figure}

\begin{figure}[H]
\centering
\includegraphics[width=0.8\textwidth]{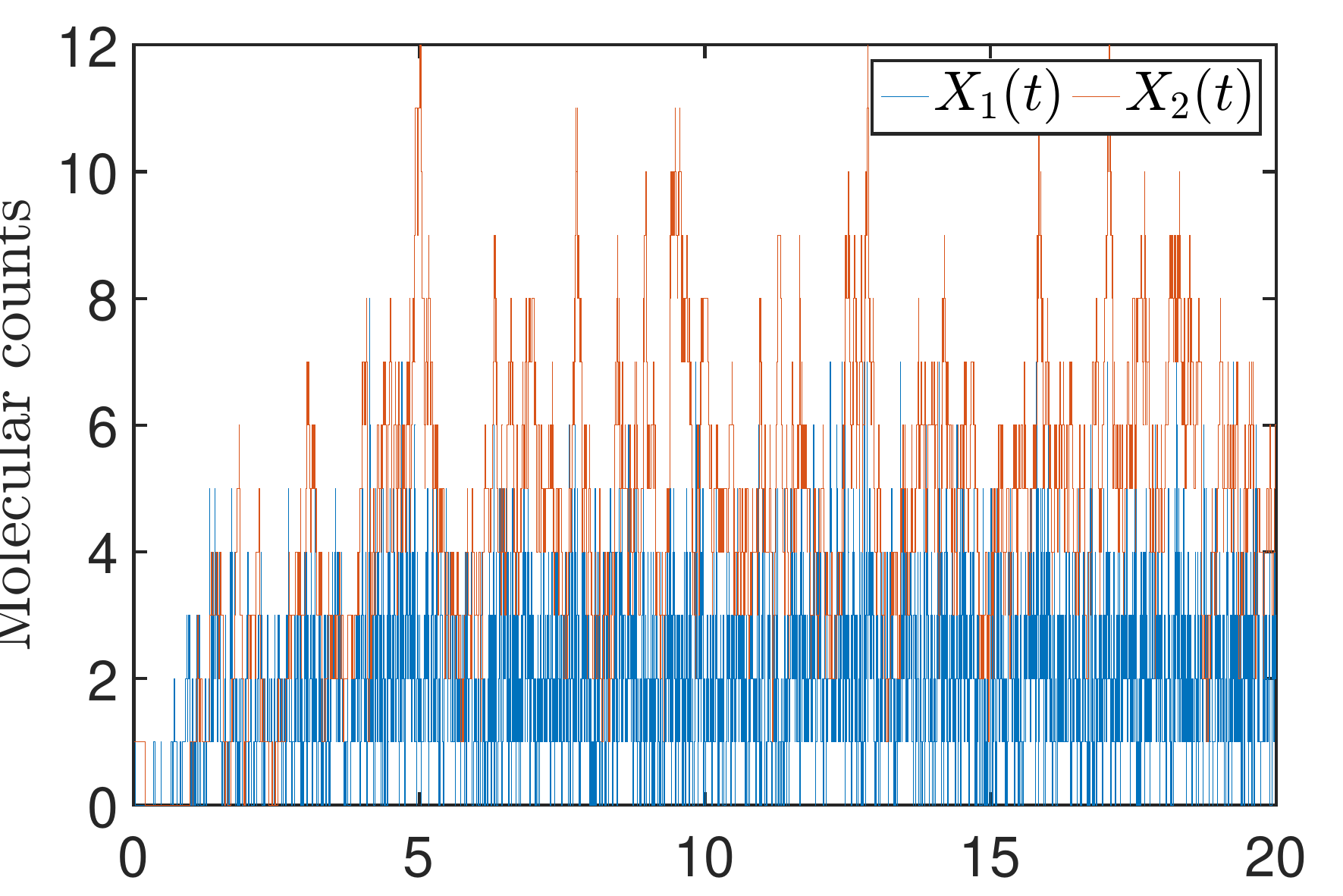}%\numfig}
  \caption{Evolution of proteins populations in a single cell (no disturbance)}\label{fig:cell}
\end{figure}

\begin{figure}[H]
\centering
\includegraphics[width=0.8\textwidth]{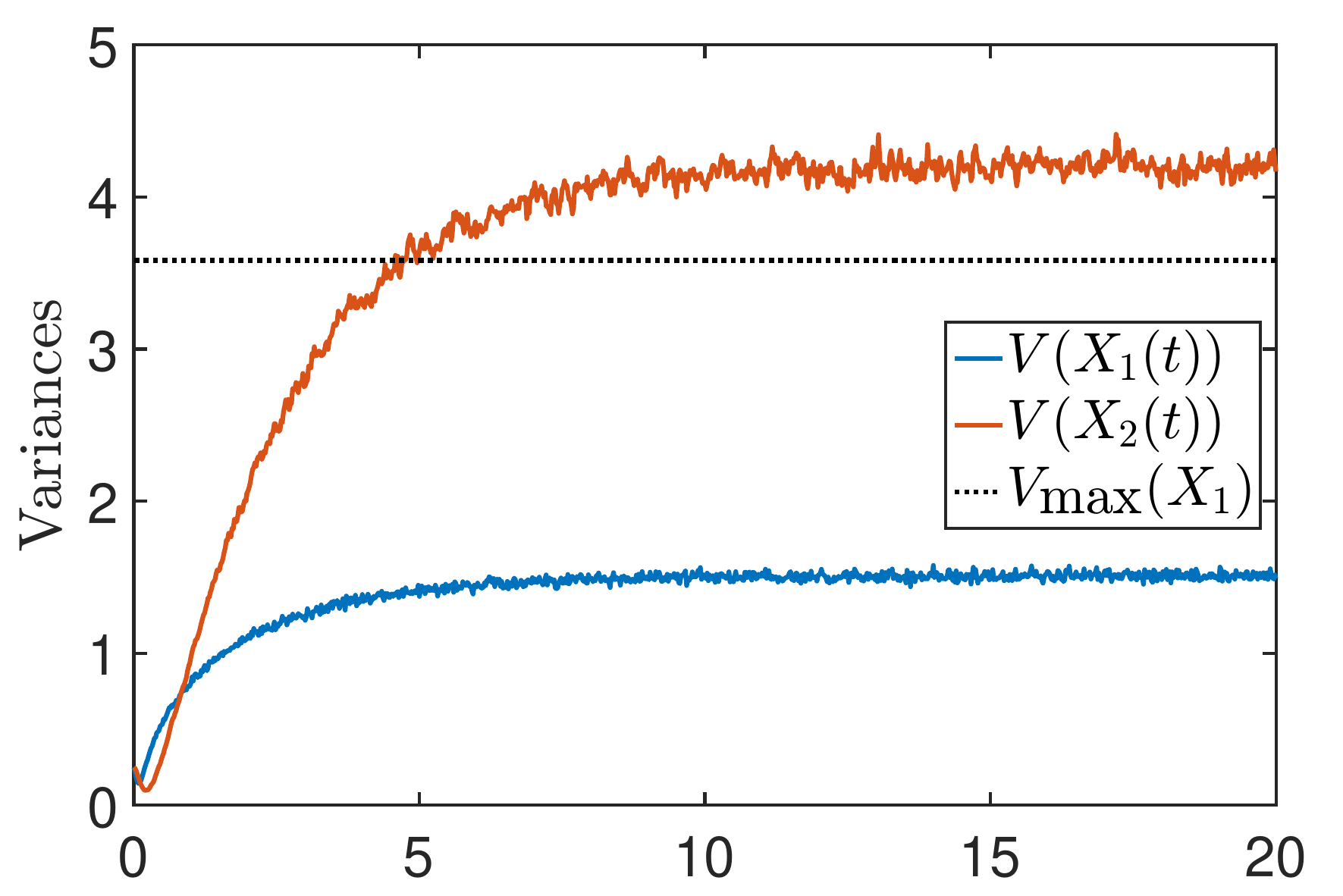}%\numfig}
  \caption{Evolution of the variances (no disturbance). }\label{fig:variance}
\end{figure}
%
%\begin{minipage}{\linewidth}
%\hfill
%\begin{minipage}[t]{0.45\textwidth}
%\begin{figure}[H]
%\centering
%\includegraphics[width=0.8\textwidth]{./MatlabCDC13/Sample_Paths.pdf}%\numfig}
%  \caption{Evolution of proteins populations in a single cell (no disturbance)}\label{fig:cell}
%\end{figure}
%\end{minipage}
%\hfill
%\begin{minipage}[t]{0.45\textwidth}
%\begin{figure}[H]
%\centering
%\includegraphics[width=0.8\textwidth]{./MatlabCDC13/Variances.pdf}%\numfig}
%  \caption{Evolution of the variances (no disturbance). }\label{fig:variance}
%\end{figure}
%\end{minipage}
%\hfill
%\end{minipage}

\begin{figure}[H]
\centering
  \includegraphics[width=0.8\textwidth]{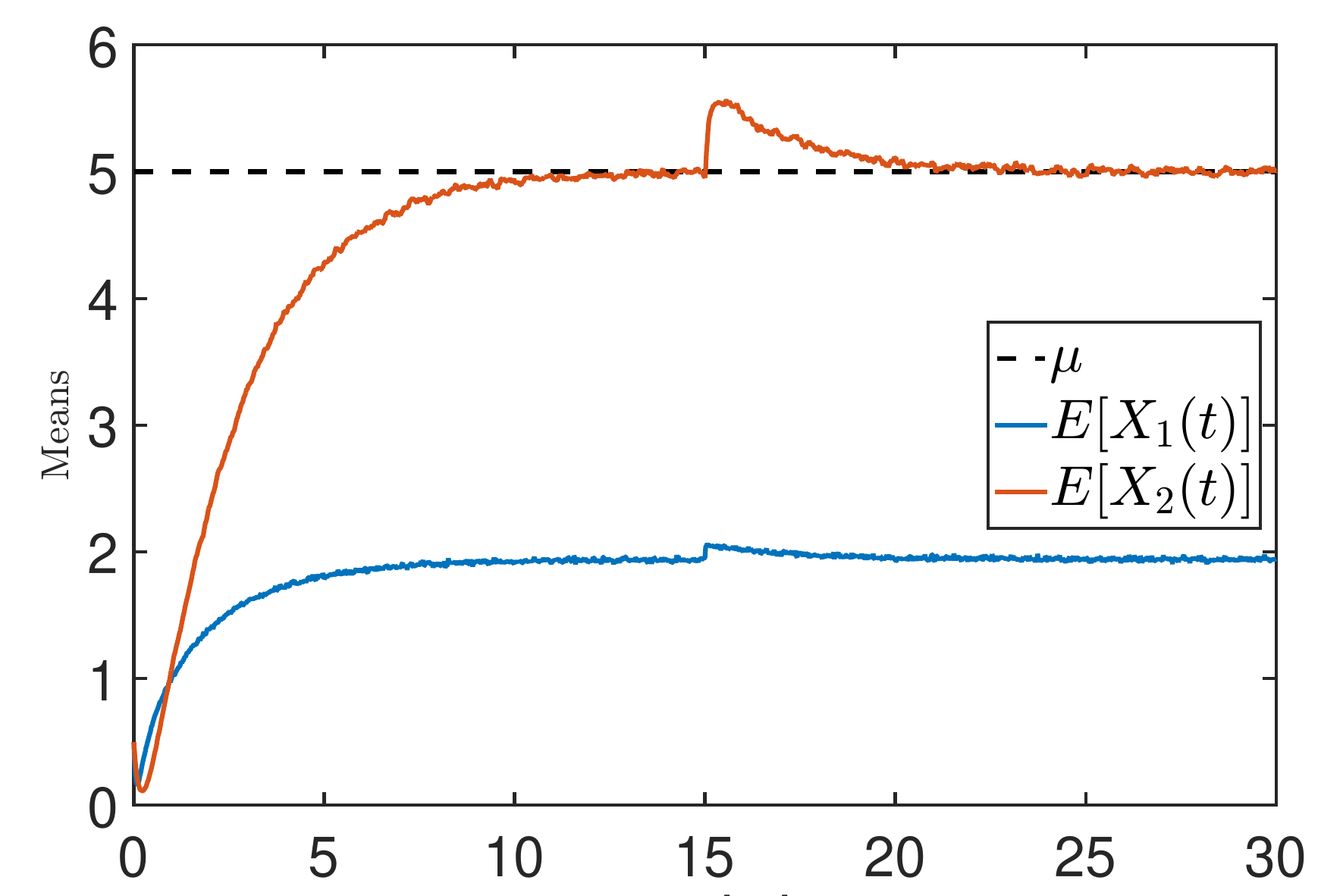}%\numfig}
  \caption{Evolution of the proteins averages in a population of 10000 cells (with input disturbance)}\label{fig:averagedu}
\end{figure}

\begin{figure}[H]
\centering
\includegraphics[width=0.8\textwidth]{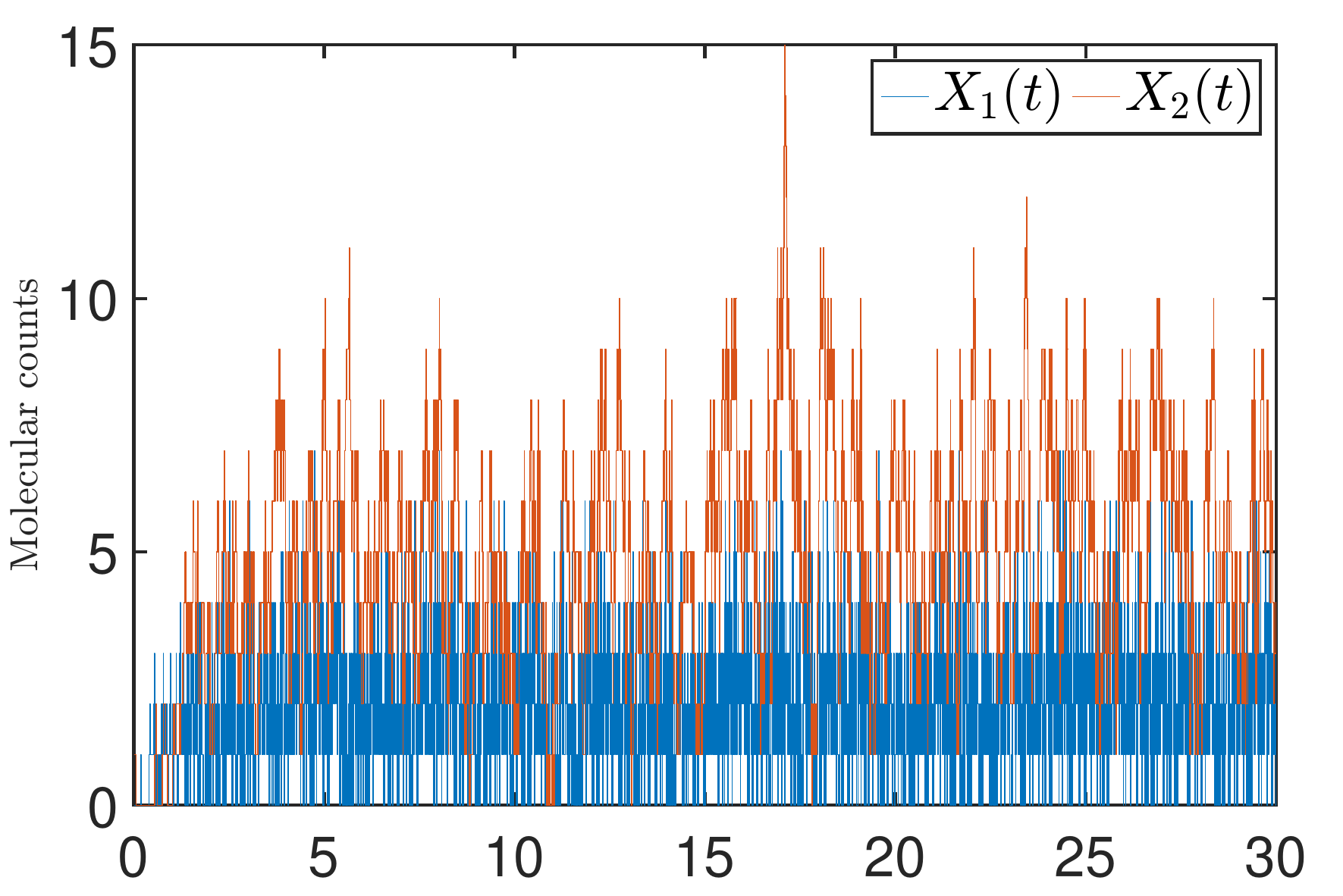}%\numfig}
  \caption{Evolution of proteins populations in a single cell (with input disturbance)}\label{fig:celldu}
\end{figure}

\begin{figure}[H]
\centering
\includegraphics[width=0.8\textwidth]{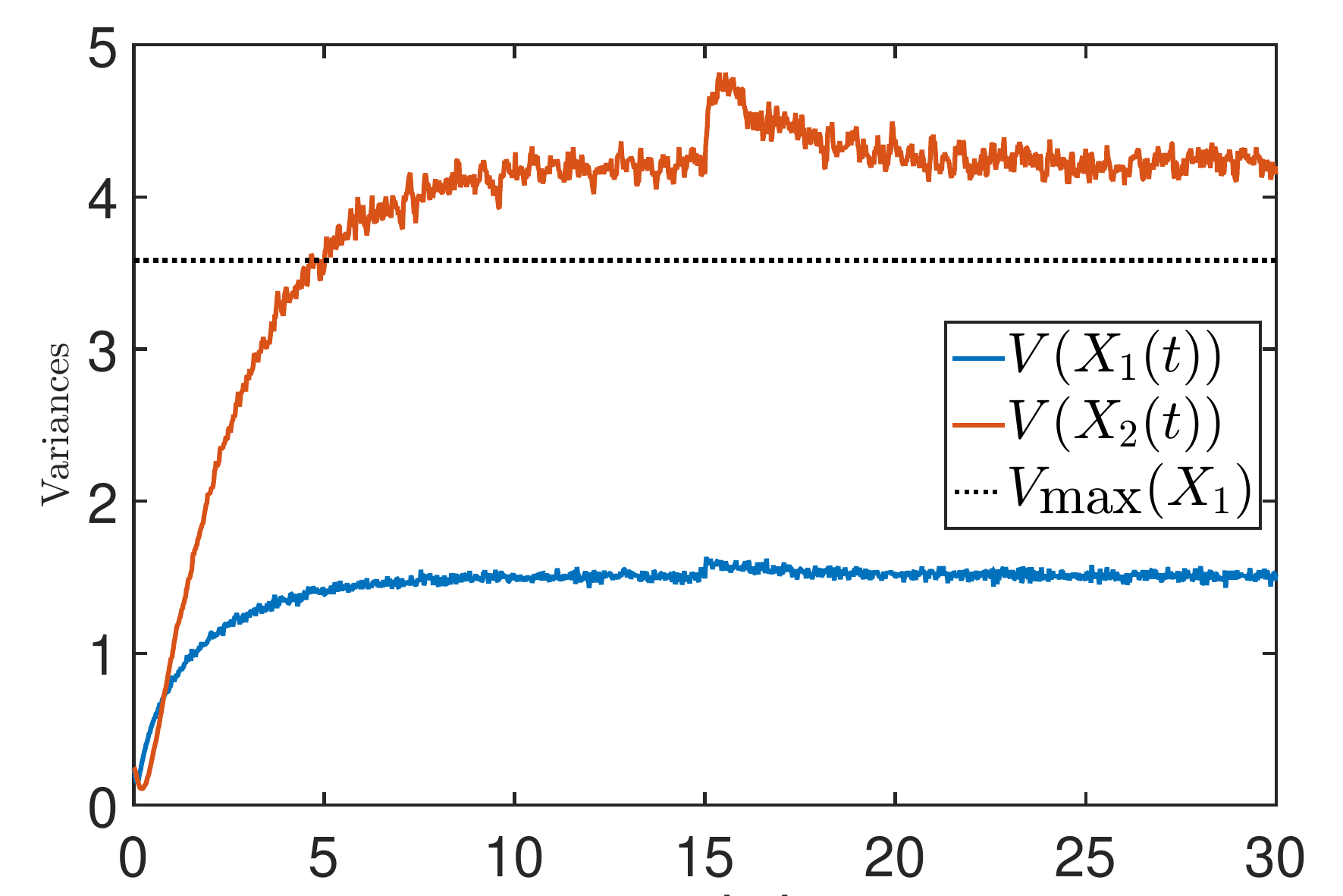}%\numfig}
  \caption{Evolution of the variances (with input disturbance).}\label{fig:variancedu}
\end{figure}

\end{document}